\documentclass[reqno,12pt]{amsart}
\usepackage{amsmath,amssymb,amsthm,graphicx,mathrsfs,bbm,url}
\usepackage{amsthm}
\usepackage{wrapfig}
\usepackage{mathtools}
\usepackage[utf8]{inputenc}
 \usepackage[T1]{fontenc}
\usepackage[usenames,dvipsnames]{color}
\usepackage[colorlinks=true,linkcolor=Red,citecolor=Green]{hyperref}
\usepackage[super]{nth}
\usepackage[open, openlevel=2, depth=3, atend]{bookmark}
\hypersetup{pdfstartview=XYZ}
\usepackage[font=footnotesize]{caption}
\usepackage{a4wide}
\usepackage[shortlabels]{enumitem}

\usepackage{epstopdf}
 
\usepackage{hyperref}

\theoremstyle{plain}
\newtheorem{theorem}{Theorem}[section]
\newtheorem*{theorem*}{Theorem}

\newtheorem{lemma}[theorem]{Lemma}
\newtheorem{proposition}[theorem]{Proposition}
\newtheorem{corollary}[theorem]{Corollary}

\theoremstyle{definition}
\newtheorem{definition}[theorem]{Definition}

\theoremstyle{remark}
\newtheorem{remark}[theorem]{Remark}

\numberwithin{equation}{section}

\newcommand{\C}{\mathbb{C}}
\newcommand{\R}{\mathbb{R}}

\newcommand{\Z}{\mathbb{Z}}

\newcommand{\E}{\mathcal{E}}

\newcommand{\V}{\mathbb{V}}
\newcommand{\X}{\mathbf{X}}

\newcommand{\HH}{\mathbb{H}}
\newcommand{\Ss}{\mathbb{S}}
\newcommand{\eps}{\varepsilon}

\newcommand{\mc}{\mathcal}

\newcommand{\dd}{\mathrm{d}}

\DeclareMathOperator{\vol}{vol}

\DeclareMathOperator{\e}{\mathbf{e}}

\DeclareMathOperator{\Hol}{Hol}
\DeclareMathOperator{\comp}{comp}

\newcommand{\be}{\begin{equation}}
\newcommand{\ee}{\end{equation}}

% ...

% ...

\title
[Isometric extensions of Anosov flows via microlocal analysis]
{Isometric extensions of Anosov flows via microlocal analysis}

\author{Thibault Lefeuvre}

\address{Université de Paris and Sorbonne Université, CNRS, IMJ-PRG, F-75006 Paris, France.}

\email{tlefeuvre@imj-prg.fr}

\begin{document}

\begin{abstract}
The aim of this note is to revisit the classical framework developed by Brin, Pesin \cite{Brin-Pesin-74, Brin-75-1,Brin-75-2} and others to study ergodicity and mixing properties of isometric extensions of volume-preserving Anosov flows, using the microlocal framework developed in the theory of Pollicott-Ruelle resonances. The approach of the present note is reinvested in a crucial way in the companion paper \cite{Cekic-Lefeuvre-Moroianu-Semmelmann-21} in order to show ergodicity of the frame flow on negatively-curved Riemannian manifolds under nearly $1/4$-pinched curvature assumption (resp. nearly $1/2$-pinched) in dimension $4$ and $4\ell+2, \ell > 0$ (resp. dimension $4\ell, \ell > 0$).
\end{abstract}

\maketitle

\tableofcontents

\newpage

\section{Introduction}

Let $M$ be a smooth closed connected manifold equipped with a smooth Anosov flow $(\varphi_t)_{t \in \R}$ generated by a vector field $X_M$, that is, such that the tangent bundle splits as a continuous flow-invariant direct sum
\[
TM = \R X_M \oplus \mathbb{E}^s_M \oplus  \mathbb{E}^u_M,
\]
and there exist constants $C, \lambda > 0$ such that
\begin{equation}
\label{equation:anosov}
\begin{array}{l}
\forall t \geq 0, \forall w \in \mathbb{E}^s_M, \qquad |\dd\varphi_t(w)| \leq Ce^{-t\lambda}|w|, \\
\forall t \leq 0, \forall w \in \mathbb{E}^u_M, \qquad |\dd\varphi_t(w)| \leq Ce^{-|t|\lambda}|w|.
\end{array}
\end{equation}
The smooth metric $|\bullet| = g(\bullet,\bullet)^{1/2}$ is arbitrary here. We will further assume that $(\varphi_t)_{t \in \R}$ is \emph{volume-preserving}, that is, it preserves a smooth \emph{probability measure} $\mu_M$.

Let $\pi : E \to M$ be a Riemannian fiber bundle with fiber isometric to $(F,g_F)$, a smooth closed oriented Riemannian manifold, see \S\ref{ssection:lift} for a formal definition. We consider an extension $(\Phi_t)_{t \in \R}$ to $E$ of the flow $(\varphi_t)_{t \in \R}$ i.e. such that
\begin{equation}
\label{equation:lift}
\pi \circ \Phi_t = \varphi_t \circ \pi,
\end{equation}
and we assume that for all $x \in M$ and $t \in \R$, $\Phi_t : E_x \to E_{\varphi_t x}$ is an isometry. We call such a flow an \emph{isometric extension of an Anosov flow}. It is a typical example of a \emph{partially hyperbolic dynamical system} and natural examples are provided by extensions of Anosov flows to principal $G$-bundles, such as the frame flow, see \cite{Brin-75-1,Brin-75-2,Brin-Gromov-80,Brin-Karcher-83,Burns-Pollicott-03} and \S\ref{ssection:principal} for a discussion on the principal $G$-bundle case. Since $X_M$ preserves a smooth measure $\mu_M$ by assumption and $(\Phi_t)_{t \in \R}$ acts by fiberwise isometries, the vector field $X_E$ preserves a canonical smooth measure $\mu_E$ on $E$ (the pullback of the measure on $M$ wedged with the measure on the fibers of $E$) and a natural question is to understand the ergodicity and the mixing properties of the flow $(\Phi_t)_{t \in \R}$ with respect to this measure. Below, the $L^2$ space on $E$ will always be understood with respect to the measure $\mu_E$.

Let $x_\star \in M$ be an arbitrary periodic point and let $\mathbf{G}$ be Parry's free monoid associated to $x_\star$, which is defined as the formal set of words made of homoclinic orbits to $x_\star$ as introduced in \cite{Cekic-Lefeuvre-21-1}, see \S\ref{section:preliminaries} for further details. We shall see below that the flow $(\Phi_t)_{t \in \R}$ induces a representation
\[
\rho : \mathbf{G} \to \mathrm{Isom}(F),
\]
where $F \simeq E_{x_\star}$ is identified with the fiber over $x_\star$. Since $\mathrm{Isom}(F)$ is a closed Lie group \cite{Myers-Steenrod-39}, the closure of the image of the representation $H := \overline{\rho(\mathbf{G})}$ is a closed Lie subgroup \cite[Theorem 2.3]{Helgason-01}. The group $H$ is well-defined up to conjugacy in $\mathrm{Isom}(F)$. Since $E$ may not be a principal bundle, the orbit space $H \backslash F$ may not be a smooth manifold but it is still a Hausdorff topological space endowed with a natural measure $\nu := \mathrm{pr}_* \mu_F$, where $\mathrm{pr} : F \to H \backslash F$ is the projection and $\mu_F$ is the Riemannian measure on $F$.\\

We will show:

\begin{theorem*}
\label{theorem:brin}
Under the above assumptions, the followings holds:
\begin{enumerate}[(a)]
\item \emph{Ergodicity:} There exists an open $H$-invariant subset $F_0 \subset F$ of full measure (with respect to $\mu_F$) such that for all $\mathfrak{p} \in H \backslash F_0$, there exists an associated flow-invariant smooth submanifold $Q({\mathfrak{p}}) \subset E$ which is a smooth Riemannian fiber bundle over $M$ with fiber diffeomorphic to a closed manifold $Q_0$ (independent of $\mathfrak{p}$) and such that the restriction of $(\Phi_t)_{t \in \R}$ to $Q(\mathfrak{p})$ is ergodic (with respect to the flow-invariant smooth measure induced by $\mu_E$ on $Q(\mathfrak{p})$). Moreover, there exists a natural isometry
\begin{equation}
\label{equation:psi}
\Psi : L^2(H \backslash F, \nu) \xrightarrow{\sim}  \ker_{L^2}(X_E).
\end{equation}
In particular, $H$ acts transitively on the fiber $F$ if and only if the flow $(\Phi_t)_{t \in \R}$ is ergodic on $E$. \\

\item \emph{Mixing:} If the fiber of $Q(\mathfrak{p})$ is not the total space of a Riemannian submersion over the circle and $(\varphi_t)_{t \in \R}$ is mixing, then the restriction of $(\Phi_t)_{t \in \R}$ to $Q(\mathfrak{p})$ is also mixing. In particular, if $H$ acts transitively on the fiber $F$ and $F$ is not the total space of a fiber bundle over the circle, then the flow $(\Phi_t)_{t \in \R}$ is mixing on $E$.
\end{enumerate}
\end{theorem*}

We formulate some remarks on Theorem \ref{theorem:brin}: 

(1) The extension isomorphism $\Psi$ in \eqref{equation:psi} is defined in \S\ref{ssection:ergodicity}, Lemma \ref{lemma:isomorphism}. It consists in ``pushing'' an $H$-invariant $L^2$-function defined on a certain fiber $E_{x_\star} \simeq F$ by the flow $(\Phi_t)_{t \in \R}$ in order to obtain a well-defined invariant function in $\ker_{L^2}(X_E)$.

(2) The manifold $Q_{0}$ is obtained as a \emph{principal orbit} for the $H$-action on the fiber $F$, see \S\ref{ssection:actions} where this is further described. The condition that $Q_0$ does not fiber over the circle is sufficient but obviously not necessary for mixing as the frame flow over a $3$-dimensional hyperbolic manifold (the frame bundle is then an $\Ss^1$-bundle over the $5$-dimensional unit tangent bundle of the manifold) has fiber isometric to the circle and is nevertheless mixing, and even exponentially mixing \cite{Howe-Moore-79, Moore-87, Guillarmou-Kuster-20}. Yet, it is simple to construct an example of an extension to an $\Ss^1$-bundle that is ergodic and not mixing: this is satisfied by the flow $\Phi_t(x,\theta) := (\varphi_t(x),\theta + t \text{ mod } 2\pi)$ on $M \times \Ss^1$ for instance, see \S\ref{ssection:mixing}, Lemma \ref{lemma:circle}. Observe that, when a manifold is connected, a necessary condition for it to fiber over the circle is that it is infinite. Compact semisimple Lie groups have finite fundamental group (see \cite[Corollary 3.9.4]{Duistermaat-Kolk-00}) so they never fiber over $\Ss^1$, which easily implies that the extension of a mixing volume-preserving Anosov flow to a principal $G$-bundle, where $G$ is a compact semisimple compact Lie group, is ergodic if and only if it is mixing, see \S\ref{ssection:principal} for a discussion. Actually, the original aim of this note was to deal with Anosov flows extended to principal bundles but since the theory is essentially the same for isometric extensions, we wrote it in this context.

(3) The volume-preserving Anosov flow $(\varphi_t)_{t \in \R}$ is \emph{not} mixing if and only if it is the suspension of an Anosov diffeomorphism by a constant roof function: this is known as the Anosov alternative, see \cite{Anosov-67,Plante-72} or Appendix \ref{appendix}.

(4) The main result is stated for Anosov dynamics but it should be clear from the proofs that it could be adapted to discrete-time dynamics or to the setting of Axiom A flows.

(5) Eventually, smoothness of the objects could be considerably weakened and it is very likely that only $C^{1+\eps}$-regularity (for some $\eps > 0$) of the vector field $X_E$ is sufficient to make the whole machinery work. The microlocal arguments in \S\ref{ssection:invariant} should then be replaced by the more recent technology developed in \cite[Sections 2 and 3]{Adam-Baladi-18}.  \\

To conclude, let us mention that some of the results of this manuscript might already be contained in the literature: For instance, Brin showed that ergodicity is equivalent to $H = G$ in the case of principal $G$-bundles \cite{Brin-Pesin-74,Brin-75-1,Brin-75-2} and Dolgopyat \cite[Corollary 4.8]{Dolgopyat-02} had already noticed (in the case of Anosov diffeomorphisms extensions) that semisimplicity of $G$ implies that ergodicity is equivalent to mixing. However, the precise structure of $\ker_{L^2} X_E$ described in \eqref{equation:psi} seems to be new. More generally, results on ergodicity for isometric extensions of hyperbolic dynamics are spread out in the literature, hard to locate, and usually not written in a modern way.

Theorem \ref{theorem:brin} is also quite far from recent considerations on partially hyperbolic dynamics \cite{Hasselblatt-Pesin-06,Wilkinson-10}, where dynamical systems may not arise from such a geometric framework. The goal of this note is to gather the arguments in one place and to rewrite them in a self-contained and concise way, making use of both tools in hyperbolic dynamics (Parry's free monoid and the transitivity group, see \S\ref{ssection:lift}) and microlocal analysis (the theory of Pollicott-Ruelle resonances, see \S\ref{ssection:invariant}). The description of Anosov flows via transfer operators and anisotropic spaces is now fairly classical and was developed in the past twenty years, see  \cite{Liverani-04, Gouezel-Liverani-06,Butterley-Liverani-07,Faure-Roy-Sjostrand-08,Faure-Sjostrand-11,Faure-Tsuji-13,Dyatlov-Zworski-16}. In particular, thanks to this analytic approach, the Hopf argument (which is classically at the root of ergodicity) may be skirted.

 Moreover, since our description of the transitivity group via Parry's free monoid has a more representation-theoretic flavour, it makes it also clearer that the (non-)ergodicity of the extended flow on principal bundles is intimately connected to the (non-)existence of reductions of the structure group of the bundle, or to flow-invariant sections on certain associated vector bundles (see Proposition \ref{proposition:isomorphism} for instance). These questions may then be addressed by means of geometric identities such as the twisted Pestov/Weitzenböck identities \cite{Guillarmou-Paternain-Salo-Uhlmann-16}, see \cite[Section 3]{Cekic-Lefeuvre-Moroianu-Semmelmann-21} where this is further discussed.
 
 In particular, this idea is used in a fundamental way in the companion paper \cite{Cekic-Lefeuvre-Moroianu-Semmelmann-21} where we show that the frame flow of nearly $1/4$-pinched manifolds (resp. nearly $1/2$-pinched) of dimension $4$ and $4 \ell + 2, \ell >0$ (resp. $4\ell$) is ergodic and mixing. In dimensions $4$ and $4 \ell + 2, \ell >0$, this almost solves a long-standing conjecture of Brin \cite[Conjecture 2.6]{Brin-82}. \\

\noindent \textbf{Notations:} For standard notions in hyperbolic dynamics (e.g. stable, unstable manifolds etc.), we refer to \cite{Hasselblatt-Katok-95, Fisher-Hasselblatt-19}. On a smooth closed manifold $M$, given a (partially) hyperbolic flow $(\varphi_t)_{t \in \R}$ generated by a smooth vector field $X \in C^\infty(M,TM)$, we will call flowpath an orbit of the flow $(\varphi_t)_{t \in \R}$ and $us$-path a continuous path in a stable or unstable leaf of $(\varphi_t)_{t \in \R}$. If $X$ preserves the smooth measure $\mu$, we will say that the flow $(\varphi_t)_{t \in \R}$ is \emph{ergodic} (with respect to $\mu$) if $u \in L^2(M,\dd\mu), Xu = 0$ implies that $u$ is constant, and \emph{mixing} if for all $u,v \in L^2(M,\dd \mu)$ with $u$ being of $0$-average (with respect to the invariant measure $\mu$), the correlation $C_t(u,v) := \int \varphi_t^* u \cdot  \overline{v} ~\dd \mu$ converges to $0$ as $t \to \infty$. \\

\noindent \textbf{Acknowledgement:} We warmly thank M. Ceki{\'c} for pointing out that an argument was incomplete in an earlier version of this manuscript, for his thorough reading and for several comments improving the presentation. We are also grateful to the anonymous referees for their comments, hopefully clarifying the exposition.

\section{Preliminaries}

\label{section:preliminaries}

In the following, we will always assume that $M$ is a smooth closed connected manifold equipped with an Anosov flow $(\varphi_t)_{t \in \R}$, with generator $X_M \in C^\infty(M,TM)$, preserving a smooth measure $\mu_M$. In this case, the flow is necessarily transitive, ergodic and it is mixing if and only if it is not the suspension of an Anosov diffeomorphism by a constant roof function, see \cite{Anosov-67,Plante-72} or Appendix \ref{appendix} for a proof.

\subsection{Isometric extensions of Anosov flows}

\label{ssection:lift}

Let $\pi : E \to M$ be a smooth fiber bundle with fiber diffeomorphic to a closed manifold $F$. We will say that $E$ is a \emph{Riemannian fiber bundle} over $M$ with fiber isometric to $(F,g_F)$ if $E$ admits a reduction of its structure group to $\mathrm{Isom}(F)$, that is, the transition functions between local trivializations of the bundle can be taken with values in $\mathrm{Isom}(F)$.
For instance, if $\mc{E} \to M$ is a Euclidean vector bundle of rank $r$ over $M$, then the unit sphere bundle $E := \left\{ (x,v) \in E ~|~ x \in M, |v|=1\right\}$ (endowed with the restriction of the metric to the sphere) is a Riemannian fiber bundle with fiber isometric to the standard round metric $(\Ss^{r-1},g_{\Ss^{r-1}})$.

\begin{definition} We say that $(\Phi_t)_{t \in \R}$ is an isometric extension of the flow $(\varphi_t)_{t \in \R}$ if
\[
\pi \circ \Phi_t = \varphi_t \circ \pi,
\]
for all $t \in \R$ and $\Phi_t : E_x \to E_{\varphi_t x}$ is an isometry for all $x \in M, t \in \R$.
\end{definition}

The flow $(\Phi_t)_{t \in \R}$ is a typical example of a \emph{partially hyperbolic flow} i.e. the tangent space $TE$ admits a continuous flow-invariant decomposition as
\[
TE = \V \oplus \mathbb{E}^s_E \oplus \mathbb{E}^u_E \oplus \R X_E,
\]
where $\V$ is the \emph{vertical} bundle (the tangent bundle to the fibers of $E$) and $d\pi(\mathbb{V}) = 0$, $X_E$ is the generator of $(\Phi_t)_{t \in \R}$ and $d\pi(X_E) = X_M$, $\mathbb{E}^{s,u}_E$ are the respective stable/unstable bundles on $E$ and satisfy $d \pi(\mathbb{E}^{s,u}_E) = \mathbb{E}^{s,u}_M$, see \cite{Burns-Pugh-Shub-Wilkinson-99}. The bundles $\mathbb{E}^{s,u}_{E,M}$ all integrate to flow-invariant stable/unstable manifolds on $M$ (resp. on $E$) which will be denoted by $W^{s,u}_{M,E}$.

Given $y \in W^{s}_M(x)$, we define the \emph{stable holonomy} $\mathrm{Hol}^s_{x \to y} : E_x \to E_y$ as the unique point of intersection
\begin{equation}
\label{equation:defhol}
w'  = \mathrm{Hol}^s_{x \to y} w := W^s_{E}(w) \cap E_y.
\end{equation}
The \emph{unstable holonomy} is defined similarly for $y \in W^u_M(x)$ and we will call \emph{flow holonomy} the map $\Phi_t : E_x \to E_{\varphi_t x}$ given by the flow itself. For fixed $x \in M, y \in W^s_M(x)$, the map $\mathrm{Hol}^s_{x \to y} : E_x \to E_y$ is a smooth isometry: indeed, it is a continuous distance-preserving map between $E_x \to E_y$ since the flow $\Phi_t$ is an isometry between fibers; by \cite{Myers-Steenrod-39, Palais-57}, it is smooth. It can also be checked that $\mathrm{Hol}^s$ depends Hölder-continuously on the basepoints $x \in M, y \in W^s_M(x)$: this is due to \eqref{equation:defhol} combined with the fact that the foliation $W^{s}_E$ is Hölder-continuous.

Alternatively, it will be convenient to have at our disposal another equivalent definition for these holonomies. For that, let $g$ be an arbitrary metric on $M$ (for instance, the same as the one defining \eqref{equation:anosov}). We also fix an arbitrary Ehresmann connection\footnote{This simply consists in choosing a smooth \emph{horizontal} subbundle $\HH \subset TE$ such that $\V \oplus \HH = TE$; such an object allows to define a \emph{parallel transport} of sections of $E$ along curves on the base $M$ on $E$, see \cite[Section 9.9]{Kolar-Michor-Slovak-93}.}. For $x,y \in M$ close enough, this allows to define a map $\tau_{x \to y} : E_x \to E_y$ by parallel transport with respect to the Ehresmann connection along the unique short geodesic joining $x$ to $y$. Given $x \in M$, we let $W^{s,u}_M(x)$ be the strong stable/unstable manifolds of $x$ in $M$. 
\begin{equation}
\label{equation:holonomy}
\mathrm{Hol}^s_{x \to y} w := \lim_{t \to +\infty} \Phi_{-t} \circ \tau_{\varphi_t x \to \varphi_t y} \circ \Phi_t w.
\end{equation}
Justification for the convergence of \eqref{equation:holonomy} is almost straightforward: there exist constants $C, \lambda > 0$ such that for all $t > 0$, the distance between $d(\varphi_t x,\varphi_t y)$ is bounded by $C e^{-\lambda t}$. This implies that the distance in $E$ between $\tau_{\varphi_t x \to \varphi_t y} \circ \Phi_t w$ and the flowline of $w'$ is also controlled by $C e^{-\lambda t}$ and thus, since $\Phi_{-t}$ acts by isometries, the distance between $w'$ and $\Phi_{-t} \circ \tau_{\varphi_t x \to \varphi_t y} \circ \Phi_t w$ is also exponentially small. In the case of an affine connection, one can also argue by using the Ambrose-Singer formula, see \cite[Section 3.2.2]{Cekic-Lefeuvre-21-1}. \\

We fix an arbitrary periodic point $x_\star \in M$ of period $T_\star$ and denote its orbit by $\gamma_\star$. We let $\mc{H}$ be the set of all \emph{homoclinic orbits} to $\gamma_\star$, i.e. the set of all orbits converging in the past and in the future to $\gamma_\star$. In particular, note that $\gamma_\star \in \mc{H}$.

\begin{lemma}
\label{lemma:density}
The union of all homoclinic orbits to $\gamma_\star$ is dense in $M$.
\end{lemma}

The proof follows from the shadowing Lemma, see \cite[Theorem 3.1]{Gouezel-Lefeuvre-19} or \cite[Lemma 3.11]{Cekic-Lefeuvre-21-1} for instance. We let $\mathbf{G}$ be the formal free monoid of elements of $\mc{H}$, that is,
\[
\mathbf{G} := \left\{ \gamma_1 ... \gamma_k ~|~ k \in \Z_{\geq 0}, \gamma_i \in \mc{H}\right\},
\]
and we call $\mathbf{G}$ \emph{Parry's free monoid}\footnote{Even if this was not formally introduced like this by Parry, the idea of considering such an object should be imputed to him \cite{Parry-99}.}, as introduced in \cite{Cekic-Lefeuvre-21-1}. Besides the identity element $\mathbf{1}_{\mathbf{G}} \in \mathbf{G}$ (formally, it is the empty word), $\mathbf{G}$ contains a particular element $\gamma_\star$, which is the orbit of the point $x_\star$ itself.

\begin{definition}
\label{definition:representation}
We define the representation of Parry's free monoid
\[
\rho : \mathbf{G} \to  \mathrm{Isom}(E_{x_\star}),
\]
by the following process: given $\gamma \in \mc{H}$, $\gamma \neq \gamma_\star$, we fix two arbitrary points $x_1(\gamma) \in W^u_M(x_\star) \cap \gamma, x_2(\gamma) \in W^s_M(x_\star) \cap \gamma$ close enough to $x_\star$ and we let
\begin{equation}
\label{equation:cross}
\rho(\gamma) := \Hol^s_{x_2 \to x_\star} \circ \Hol^c_{x_1 \to x_2} \circ \Hol^u_{x_\star \to x_1}.
\end{equation}
For $\gamma_\star \in \mathbf{G}$, we set $\rho(\gamma_\star) := \Phi_{T_\star}({x_\star})$.
\end{definition}

This definition \emph{depends} on a choice of points $x_1(\gamma),x_2(\gamma)$ for every homoclinic orbit $\gamma$. 

\begin{definition}
\label{definition:transitivity}
We define the \emph{transitivity group} as $H := \overline{\rho(\mathbf{G})} \leqslant \mathrm{Isom}(E_{x_\star})$.
\end{definition}

Although it is claimed in the definition, it is not immediate that $H$ is a group since $\mathbf{G}$ is a monoid.

\begin{lemma}
$H$ is a closed subgroup of $\mathrm{Isom}(E_{x_\star})$. Moreover, it is independent of the choice of points $x_1(\gamma)$ and $x_2(\gamma)$ in Definition \ref{definition:representation}.
\end{lemma}

Observe that, being a closed subgroup of the compact Lie group $\mathrm{Isom}(E_{x_\star}) \simeq \mathrm{Isom}(F)$, $H$ is also a compact Lie group \cite[Theorem 2.3]{Helgason-01}.

\begin{proof}
For the group characterization, the only non-trivial property to check is that given $h \in H, h^{-1} \in H$. For that, it suffices to check that given $g \in \mathbf{G}$ $\rho(g)^{-1} \in H$. But since $\mathrm{Isom}(F)$ is compact, there exists a sequence $(n_k)_{k \in \Z_{\geq 0}}$ such that $\rho(g)^{n_k} = \rho(g^{n_k}) \to_{k \to \infty} \rho(g)^{-1}$, which proves the claim.

We now consider an homoclinic orbit $\gamma$ and pairs of points $x_1,x_2$ and $x_1', x_2'$ as in Definition \ref{definition:representation}. We let $\rho(\gamma)$ and $\rho'(\gamma)$ be the respective elements of $\mathrm{Isom}(E_{x_\star})$ obtained for this choice of points \eqref{equation:cross}. We observe that there exists $k_1,k_2 \in \Z$ such that $x_i = \varphi_{k_i T_\star}x_i'$, for $i=1,2$. This implies that
\[
\rho'(\gamma) = \rho(\gamma_\star)^{-k_2} \rho(\gamma) \rho(\gamma_\star)^{k_1} \in H,
\]
and this proves the claim.
\end{proof}

We conclude by the following important remark:

\begin{remark}
\label{remark:choice}
In practice, we will always choose an arbitrary isometry $\varphi_\star : F \to E_{x_\star}$ and then Parry's representation in Definition \ref{definition:representation} gets identified via $\varphi_\star$ to a representation $\rho : \mathbf{G} \to \mathrm{Isom}(F)$. Changing the isometry $\varphi_\star$ by $\varphi_\star'$, one gets another conjugate representation $\rho' : \mathbf{G} \to \mathrm{Isom}(F)$. The transitivity group $H$ can then be seen as a subgroup of $\mathrm{Isom}(F)$ but the same remark applies: changing the isometry $F \to E_{x_\star}$ by another one, one gets a conjugate group $H' \leqslant \mathrm{Isom}(F)$.
\end{remark}

\subsection{Invariant sections on vector bundles}

\label{ssection:invariant}

Let $\mc{E} \to M$ be a smooth real/complex vector bundle over $M$ equipped with a Euclidean/Hermitian metric. Let $\X : C^\infty(M,\mc{E}) \to C^\infty(M,\mc{E})$ be a differential operator acting as a derivation in the $X_M$-direction i.e. such that for all $f \in C^\infty(M), u \in C^\infty(M,\mc{E})$:
\begin{equation}
\label{equation:derivation}
\X(f u) = (X_Mf) u + f \X u.
\end{equation}
Such an operator defines a natural parallel transport of sections of $\E$ along the flowlines of $(\varphi_t)_{t \in \R}$. The propagator $e^{t \X} : C^\infty(M,\E) \to C^\infty(M,\E)$ is defined as follows: for $f \in C^\infty(M,\E)$, $x \in M$, $(e^{t\X}f)(x)$ is the parallel transport of $f(\varphi_{-t}x)$ along the segment $(\varphi_sx)_{s \in [-t,0]}$ with respect to $\X$.

In the following, we will furher assume that $e^{t\X}$ is a fiberwise isometry. Typical examples are provided by operators $\X := \nabla^{\E}_{X_M}$, where $\nabla^{\E}$ is an arbitrary unitary\footnote{That is, parallel transport with respect to $\nabla^{\E}$ preserves the norm on $\E$.} connection on $\E$. Since $X_M$ preserves a smooth measure $\mu_M$ by assumption, we can consider the $L^p(M,\dd\mu_M)$ spaces on $M$ for $p \in [1,+\infty]$ defined with respect to that measure and the spaces $L^p(M,\E)$, with norm given by
\[
\|u\|_{L^p(M,\E)}^p := \int_M |u(x)|^p_{\E_x} \dd \mu_M(x).
\]
We introduce for $\Re(z) > 0$ the positive and negative resolvents
\begin{equation}
\label{equation:resolvent}
\begin{split}
& \mathbf{R}_+(z) := (-\X-z)^{-1} = - \int_0^{+\infty} e^{-t(\X+z)} \dd t, \\
&  \mathbf{R}_-(z) := (\X-z)^{-1} := - \int_{0}^{+\infty} e^{-t(-\X+z)} \dd t.
\end{split}
\end{equation}
Since the propagator is unitary, they satisfy
\begin{equation}
\label{equation:bound}
\|\mathbf{R}_{\pm}(z)\|_{L^p \to L^p} \leq 1/\Re(z),
\end{equation}
for all $p \in [1,+\infty]$ and $z \in \C$ such that $\Re(z) > 0$. Moreover, we have
\begin{equation}
\label{equation:adjoint}
\mathbf{R}_+(z)^* = \mathbf{R}_-(\overline{z}),
\end{equation}
for $\Re(z) > 0$, where the adjoint is taken with respect to the $L^2$-scalar product, namely,
\[
\forall f_1,f_2 \in C^\infty(M,\E), \qquad \langle \mathbf{R}_+(z)f_1,f_2\rangle_{L^2} = \langle f_1, \mathbf{R}_-(\overline{z}) f_2 \rangle_{L^2}.
\]
This follows from the fact that $\X^*=-\X$ which, in turn, follows from the fact that $e^{t\X}$ is a fiberwise isometry. We refer to \cite[Section 2.5]{Cekic-Lefeuvre-21-1} for further details. The following regularity statement will be used several times in the rest of the paper.

\begin{proposition}
\label{proposition:regularity}
Let $u \in L^1(M,\mc{E})$ such that $\X u = i \lambda u$, for some $\lambda \in \R$. Then $u$ is smooth.
\end{proposition}

Proposition \ref{proposition:regularity} will be mostly used with $u \in L^2$. The proof is based on the microlocal theory of Pollicott-Ruelle resonances on anisotropic Sobolev spaces. 


\begin{proof}
It suffices to argue in the case $\lambda=0$ since it amounts to changing $\X$ by $\X-i\lambda$, which has the same analytic properties. By \cite{Faure-Sjostrand-11,Dyatlov-Zworski-16}, there exists a scale of \emph{anisotropic Sobolev spaces} $\mc{H}_\pm^s$ and a constant $C  >0$ such that $\mathbf{R}_\pm(z) : \mc{H}_\pm^s \to \mc{H}_\pm^s$ admits a meromorphic extension from $\left\{\Re(z) > 0 \right\}$ to $\left\{\Re(z)> -Cs\right\}$. These spaces can be designed so that $\mc{H}_+^s \cap \mc{H}_-^s = H^s(M,\mc{E})$, see \cite{Guillarmou-17-1} for instance.

The meromorphic family of operators $z \mapsto \mathbf{R}_+(z) \in \mc{L}(\mc{H}^s_+)$ may have a pole at $z=0$. If it is the case, this pole is necessarily of order $1$ by \eqref{equation:bound}. We denote by $\Pi_\pm$ the spectral projector onto $\ker \X|_{\mc{H}^s_\pm}$ given by
\[
\Pi_{\pm} = -\dfrac{1}{2 i \pi} \int_{\gamma} \mathbf{R}_{\pm}(z) \dd z
\]
where $\gamma$ is a small contour around $0$, see \cite[Section 2.5]{Cekic-Lefeuvre-21-1}. (Note that this may be $0$ if there is no pole at $z=0$.) As a consequence, $z \mathbf{R}_+(z) \to_{z \to 0} - \Pi_+$ in operator norm of $\mc{L}(\mc{H}^s_+)$ and we can write
\[
z  \mathbf{R}_+(z) = -\Pi_+ + z \mathbf{R}_+^{\mathrm{hol}}(z), \qquad z\mathbf{R}_-(z) = -\Pi_- + z\mathbf{R}_-^{\mathrm{hol}}(z)
\]
where $z \mapsto \mathbf{R}_\pm^{\mathrm{hol}}(z) \in \mc{L}(\mc{H}^s_\pm)$ is holomorphic near $z=0$. Observe that $\Pi_+^*=\Pi_-$ (where the adjoint is taken with respect to the $L^2$-scalar product as before) because
\[
\Pi_+^* = +\dfrac{1}{2 i \pi} \int_{\gamma} \mathbf{R}_{+}(z)^* \dd z = +\dfrac{1}{2 i \pi} \int_{\gamma} \mathbf{R}_{-}(\overline{z}) \dd z = -\dfrac{1}{2i\pi} \int_{\gamma}  \mathbf{R}_{-}(z) \dd z = \Pi_-.
\]
Note that we used here that the relation \eqref{equation:adjoint} holds for all $z \in \C$ not being a pole of $\mathbf{R}_\pm$.

We now let $u \in L^1(M,\E)$ such that $\X u = 0$. Consider a sequence $(u_n)_{n \in \Z_{\geq 0}}$ such that each $u_n$ is smooth and $u_n \to u$ in $L^1(M,\E)$. Define $v_n := \Pi_+ u_n \in \ker \X|_{\mc{H}^s_+}$. We claim that $(\|v_n\|_{\mc{H}^s_+})_{n \in \Z_{\geq 0}}$ is bounded. We argue by contradiction. If it is unbounded, then up to a subsequence, we can assume that $\|v_n\|_{\mc{H}^s_+} \to \infty$. Since $\ker \X|_{\mc{H}^s_+}$ is a finite-dimensional vector space, up to another extraction, we then get that $v_n/\|v_n\|_{\mc{H}^s_+} \to v_\infty$ in $\mc{H}^s_+$, for some $v_\infty \in \mc{H}^s_+$ and $\|v_\infty\|_{\mc{H}^s_+} = 1$.

We will now take $z > 0$ to be a positive real number. Taking a smooth test function $\psi \in C^\infty(M,\E)$, we then get for $z > 0,n \in \Z_{\geq 0}$:
\[
\begin{split}
\langle v_\infty, \psi \rangle & = \langle v_\infty - \tfrac{v_n}{\|v_n\|_{\mc{H}^s_+}} , \psi\rangle + \tfrac{1}{\|v_n\|_{\mc{H}^s_+}} \langle \Pi_+ u_n, \psi \rangle \\
& = \langle v_\infty - \tfrac{v_n}{\|v_n\|_{\mc{H}^s_+}}, \psi \rangle + \tfrac{1}{\|v_n\|_{\mc{H}^s_+}} \langle u_n, \Pi_- \psi \rangle \\
&  =  \langle v_\infty- \tfrac{v_n}{\|v_n\|_{\mc{H}^s_+}}, \psi \rangle + \tfrac{1}{\|v_n\|_{\mc{H}^s_+}} \langle u_n, (-z\mathbf{R}_-(z) + z\mathbf{R}^{\mathrm{hol}}_-(z)) \psi \rangle.
\end{split}
\]
Observe that using \eqref{equation:bound}, we get for $z > 0$ and $n$ large enough:
\[
|\langle u_n ,-z \mathbf{R}_-(z) \psi \rangle| \leq \|u_n\|_{L^1} \|z \mathbf{R}_-(z) \psi \|_{L^\infty} \leq 2 \|u\|_{L^1} \|\psi \|_{L^\infty}.
\]
We fix $\eps > 0$. Taking $n$ large enough, we can ensure that
\[
|\langle v_\infty - \tfrac{v_n}{\|v_n\|_{\mc{H}^s_+}}, \psi \rangle| \leq \eps, ~~~~ \|v_n\|^{-1}_{\mc{H}^s_+} \leq \eps,
\]
(the first inequality follows from convergence of $v_n/\|v_n\|_{\mc{H}^s_+} \to v_\infty$ as distributions while the second follows from $\|v_n\|_{\mc{H}^s_+} \to \infty$), which implies that 
\[
|\langle v_\infty, \psi \rangle| \leq C \eps + \left|\tfrac{z}{\|v_n\|_{\mc{H}^s_+}} \langle u_n, \mathbf{R}^{\mathrm{hol}}_-(z) \psi \rangle \right|,
\]
where $C > 0$ is independent of $n$. Then, taking $z > 0$ small enough (depending on $\eps > 0$ and $n$), using that $\|\mathbf{R}^{\mathrm{hol}}_-(z)\|_{\mc{L}(\mc{H}^s_-)} \leq C$ is uniformly bounded by a constant $C > 0$ for $z$ close to $0$, we get $\langle v_\infty, \psi \rangle \leq C \eps$, which implies that $v_\infty = 0$. But this contradicts $\|v_\infty\|_{\mc{H}^s_+} = 1$. We thus conclude that $(\|v_n\|_{\mc{H}^s_+})_{n \in \Z_{\geq 0}}$ is bounded.

Once again, using that $\ker \X|_{\mc{H}^s_+}$ is a finite-dimensional space, we obtain up to a subsequence that $v_n \to v_\infty$ in $\mc{H}^s_+$. We claim that $u = v_\infty$. Indeed, fix $\eps > 0$. Using that for $z > 0$, $-z \mathbf{R}_+(z) u = u$ we get that for all test functions $\psi \in C^\infty(M,\E)$:
\[
\begin{split}
\langle u, \psi \rangle_{L^2(M,\E)}  &= \langle -z \mathbf{R}_+(z) u, \psi \rangle_{L^2(M,\E)} \\
& = \langle -z\mathbf{R}_+(z) (u-u_n), \psi \rangle + \langle -z\mathbf{R}_+(z) u_n - v_n, \psi \rangle + \langle v_n-v_\infty,\psi \rangle + \langle v_\infty, \psi \rangle.
\end{split}
\]
Observe that by \eqref{equation:bound}, we have
\[
|\langle -z\mathbf{R}_+(z) (u-u_n), \psi \rangle| \leq \|u-u_n\|_{L^1}\|\psi\|_{L^\infty} \leq \eps,
\]
for $n$ large enough, $z > 0$ and also $|\langle v_n-v_\infty,\psi \rangle| \leq \eps$ as $v_n \to v_\infty$ in $\mc{D}'(M)$. Moreover
\[
|\langle -z\mathbf{R}_+(z) u_n - v_n, \psi \rangle| \leq C \|-z \mathbf{R}_+(z)-\Pi_+\|_{\mc{L}(\mc{H}^s_+)} \|u_n\|_{\mc{H}^s_+} \|\psi\|_{(\mc{H}^s_+)'} \leq \eps,
\]
for $z > 0$ small enough (depending on $n$), since $-z \mathbf{R}_+(z)-\Pi_+ = -z\mathbf{R}_+^{\mathrm{hol}}(z)$ is small in the operator norm of $\mc{L}(\mc{H}^s_+)$. (Here $(\mc{H}^s_+)'$ is the dual space to $\mc{H}^s_+$.) This implies that $\langle u, \psi \rangle = \langle v_\infty,\psi \rangle$ for all $\psi \in C^\infty(M,\E)$, that is, $u = v_\infty \in \mc{H}^s_+$. The same argument applies with the negative resolvent $\mathbf{R}_-(z)$ and shows that $u \in \mc{H}^s_-$. Since $\mc{H}^s_+ \cap \mc{H}^s_- = H^s$, we get $u \in H^s$ and since $s > 0$ is arbitrary, we get that $u$ is smooth.
\end{proof}

As in \S\ref{ssection:lift} and Definition \ref{definition:representation}, there is a natural unitary representation
\[
\rho : \mathbf{G} \to \mathrm{U}(\E_{x_\star})
\]
obtained by taking parallel transport along homoclinic orbits with respect to $\X$. Let
\[
\E^\rho_{x_\star} := \left\{f \in \E_{x_\star} ~|~  \rho(g)f = f, \forall g \in \mathbf{G}\right\},
\]
be the set of vectors that are invariant by the representation. We have the following, see \cite[Theorem 3.5]{Cekic-Lefeuvre-21-1} and \cite[Lemma 3.6]{Cekic-Lefeuvre-21-1}:

\begin{proposition}
\label{proposition:isomorphism}
The evaluation map
\[
\mathrm{ev}_{\star} : \ker \X \cap C^{\infty}(M,\E) \to \E^\rho_{x_\star}, ~~~~ u \mapsto u(x_\star)
\]
is an isomorphism. In other words, there is a one-to-one correspondance between elements fixed by the representation and smooth flow-invariant sections.
\end{proposition}

\begin{proof}[Idea of proof]
It is straightforward to check that the map is well-defined. Injectivity is also easy to obtain since, if $u \in \ker \X \cap C^{\infty}(M,\E)$, one has $X_M |u|^2 = 0$ and thus $|u|$ is constant by ergodicity of the flow $(\varphi_t)_{t \in \R}$ (see Appendix \ref{appendix} for instance). Hence, if $\mathrm{ev}_\star(u) = u(x_\star) = 0$, we deduce that $u=0$. Surjectivity is less easy to obtain and we refer to \cite[Lemma 3.6]{Cekic-Lefeuvre-21-1} for a detailed proof. The idea is that, given $u_\star \in \mc{E}^\rho_{x_\star}$, one can construct by hand a Lipschitz-continuous section $u$ on $M$ such that $u(x_\star)=u_\star$ (by ``pushing'' $u_\star$ by the flow along homoclinic orbits). Using Proposition \ref{proposition:regularity}, we then bootstrap to a smooth section.
\end{proof}

\subsection{Isometric actions on closed manifolds}

\label{ssection:actions}

The group $H$ acts on the fiber $E_{x_\star} \simeq F$ (endowed with the smooth metric $g_F$) by isometries so it falls into the realm of isometric group actions, see \cite{Duistermaat-Kolk-00, Alexandrino-Bettiol-15} for a review on this topic. Note that the action by isometries is on the left. The quotient space $H \backslash F$ is in general not a smooth manifold but it is a topological Hausdorff space, due to the properness of the action (this is always the case for compact Lie groups acting on closed manifolds). The topology on the quotient space is the standard one making the projection $\mathrm{pr} : F \to H \backslash F$ continuous. Given $p \in F$, we define $H \cdot p$ to be the orbit of $p$. This is a smooth embedded submanifold in $F$. We let $H_p$ be the isotropy group of $p$, namely the closed subgroup of elements of $H$ fixing $p$. Obviously, the subgroup $H_p$ is constant modulo conjugacy in $H$ along the orbit $H \cdot p$ but it may vary as the point $p$ moves tranversaly to the orbits of $H$. Given $p \in F$, we let $N_p \subset T_p F$ be the normal vector space at $p$ to the orbit $H \cdot p$ (with respect to the metric $g_F$):

\begin{definition}[Slice]
\label{definition:slice}
A \emph{slice} at $p \in F$ is an embedded submanifold $S_p$ in $F$ containing $p$ such that
\begin{enumerate}
\item $T_p F = T_p (H \cdot p) \oplus T_p S_p$ and for all $x \in S_p$, $T_x F = T_x (H \cdot x) + T_x S_p$, 
\item $S_p$ is invariant by $H_p$, that is, for all $h \in H_p$, $h S_p = S_p$,
\item If $h \in H, x \in S_p$ and $h x \in S_p$, then $h \in H_p$.
\end{enumerate}
In particular, $H_x \leqslant H_p$ for all $x \in S_p$ by (3).
\end{definition}

For $\eps > 0$ small enough, let $N^{< \eps}$ be the normal bundle over the orbit $H \cdot p$, whose fiber at $x \in H \cdot p$ is given by $\left\{ \xi \in N_x ~|~ |\xi| < \eps \right\}$, where $N_x$ is the orthogonal at $x$ to the tangent space of the $H \cdot p$ orbit. There is a well-defined $H$-action on $N^{< \eps}$ given by $h(x,\xi) := (hx,dh(\xi))$, where $h \in H, x \in H \cdot p, \xi \in N_x^{< \eps}$ (it is well-defined since $H$ acts by isometries).
The following slice Theorem (see \cite[Theorem 2.4.1]{Duistermaat-Kolk-00}) plays an important role in the description of the orbit space of the $H$-action.

\begin{theorem}
\label{theorem:slice}
For $\eps > 0$ small enough, the image
\begin{equation}
\label{equation:exp}
S_p := \left\{\exp_p(\xi) ~|~ \xi \in N_p, |\xi| < \eps\right\}
\end{equation}
under the exponential map is a slice at $p$ for the $H$-action. Letting $~\mc{U} := H \cdot S_p$ be an $H$-invariant tubular neighborhood of $H \cdot p$, the exponential map
\begin{equation}
\label{equation:exp}
\exp : N^{< \eps} \to  \mc{U}, ~~~ (x,\xi) \mapsto \exp_x(\xi),
\end{equation}
is an $H$-equivariant diffeomorphism, that is, 
\begin{equation}
\label{equation:slice2}
\forall h \in H, x \in H\cdot p, \xi \in N_x^{< \eps}, \qquad  h \exp_x(\xi) = \exp_{hx}(dh_x \xi).
\end{equation}
\end{theorem}

There is a natural \emph{slice representation} (or isotropy representation) defined as
\begin{equation}
\label{equation:slice-rep}
\rho_p : H_p \to \mathrm{O}(N_p), ~~ \rho_p(h) := dh_p,
\end{equation}
mapping $h$ to its differential at the point $p$. Moreover, as a particular case of \eqref{equation:slice2}, the exponential map intertwines the $H_p$-action on $S_p$ with the action on $N_p$ via $\rho_p$, i.e. $h\exp_p(\xi)=\exp_p(\rho_p(h)\xi)$. By the slice Theorem \ref{theorem:slice}, any $H$-invariant object defined on $\mc{U}$ (such as functions, sets, measures, etc.) corresponds bijectively via $\exp$ to $H_p$-invariant objects on $N^{< \eps}$.

We have the following Fubini-type formula:

\begin{lemma}
\label{lemma:fubini}
Let $\mu_F$ be the smooth $H$-invariant Riemannian measure on $F$. Given $f \in C^\infty_{\comp}(\mc{U})$, we have
\begin{equation}
\label{equation:fubini}
\int_{\mc{U}} f(z) \dd \mu_F(z) = \int_{H \cdot p} \left( \int_{N_x^{< \eps}} (\exp^*f)(x,\xi) \nu(x,\xi) \dd \omega_{N_x}(\xi) \right) \dd \mu_{H \cdot p}(x), 
\end{equation}
where $\dd \omega_{N_x}$ is the Euclidean measure on $N_x^{< \eps}$, $\nu$ is some smooth $H$-invariant positive function on $\mc{U}$, and $\dd \mu_{H \cdot p}$ is the Riemannian measure on the submanifold $H \cdot p$ (induced by the metric on $F$). In particular, if $f$ is $H$-invariant, then
\begin{equation}
\label{equation:fubini2}
\int_{\mc{U}} f(z) \dd \mu_F(z) = \vol(H \cdot p) \int_{N_p^{< \eps}} (\exp^*f)|_{N_p}(\xi) \nu(p,\xi) \dd \omega_{N_p}(\xi).
\end{equation}
This also holds for measurable functions.
\end{lemma}

\begin{proof}
Using the exponential map \eqref{equation:exp}, we can write for $(x,\xi) \in N^{< \eps}$, $\exp^* \mu_F(x,\xi) = \nu(x,\xi) \dd \omega_{N_x}(\xi) \otimes \dd \mu_{H\cdot p}(x)$ for some smooth positive function $\nu \in C^\infty(N^{< \eps})$. Using, the $H$-equivariance of $\exp$ in \eqref{equation:slice2} and the $H$-invariance of both $\omega_{N_x}$ and $\mu_{H \cdot p}$, we easily get that $\nu$ is $H$-invariant. This proves \eqref{equation:fubini}.

For \eqref{equation:fubini2}, if $f$ is $H$-invariant, $x \in H \cdot p$ and $h \cdot p = x$, then we get, using $h_*(\nu(p,\bullet) \omega_{N_p}) = \nu(x,\bullet) \omega_{N_x}$ and $h_*(\exp^*f)|_{N_p} = (\exp^*f)|_{N_x}$, that:
\[
\int_{N_x} (\exp^*f)(x,\xi) \nu(x,\xi) \dd \omega_{N_x}(\xi) = \int_{N_p} (\exp^*f)(p,\xi) \nu(p,\xi) \dd \omega_{N_p}(\xi).
\]

\end{proof}

As a consequence of Lemma \ref{lemma:fubini}, we get that a measurable $H$-invariant set $A \subset \mc{U}$ has full measure in $\mc{U}$ if and only if the $H_p$-invariant set $A \cap N_p^{< \eps}$ has full measure in $N_p^{< \eps}$. \\

We will say that two points $x,y \in F$ are on the same \emph{orbit type} if $H_x$ is conjugate to $H_y$ (within $H$). By definition, a point $x \in F$ is said to be \emph{principal} (or $H$-principal) if all points $y$ in its neighborhood are on the same orbit type as $x$. There exists an open dense ($H$-invariant) set of points $F_0 \subset F$ with same orbit type, called the \emph{principal orbit type}, see \cite[Chapter 2, Section 8]{Duistermaat-Kolk-00}. In the following, we will refer to these as \emph{principal orbits} and points on the orbits will be called \emph{principal points}. Moreover, it is straightforward to check that $p \in F_0$ if and only if the slice representation $\rho_p$ is trivial or, equivalently, for all $x \in S_p$, $H_x = H_p$. When $F$ is compact, there is only a finite number of orbit types. We will eventually need the following:

\begin{lemma}
\label{lemma:measure}
The set of principal points $F_0 \subset F$ has full Lebesgue measure.
\end{lemma}

\begin{proof}
We argue by induction on the dimension of the space $F$. First of all, if $\dim F = 1$, then $F=\Ss^1$ and the effective isometric action is given by $H=\Ss^1$ or by some finite cyclic $H$. In both cases the statement is trivial.

We now consider a smooth closed manifold $F$ of dimension $> 1$. Let $p \in F$ be an arbitrary point, $S_p$ an exponential slice at $p$  as in \eqref{equation:exp} and $\mc{U} := H \cdot S_p$ a small $H$-invariant tubular neighborhood of $H \cdot p$. It suffices to show that $F_0 \cap \mc{U}$ has full measure since we can cover $F$ by a finite number of such $H$-invariant open neighborhoods. By Lemma \ref{lemma:fubini}, it suffices to show that $F_0 \cap N_p^{< \eps}$ has full measure in $N_p^{< \eps}$.

We can now consider the isometric action of $H_p$ on $\Ss_p^{\eps} := \left\{ \xi \in N_p ~|~ |\xi|=\eps\right\}$.
Observe that it suffices to show that principal points for the $H_p$-action on $\Ss_p^{\eps}$ have full Lebesgue measure. Indeed, if $\xi \in N_p^{\leq \eps}$ is principal for the $H_p$-action, then it is in particular principal for the $H$-action and so is the $H_p$-orbit of $\xi$. In other words, $\xi$ corresponds to a principal orbit in $\mc{U}$. Moreover, by Kleiner's lemma (see \cite[Lemma 3.70]{Alexandrino-Bettiol-15}), the half-line $\left\{t \xi ~|~ t \in (0,1]\right\}$ also consists of $H$-principal points. Thus, if $H_p$-principal points have full measure on $\Ss_p^{\eps}$, then $F_0 \cap N_p^{< \eps}$ has full measure in $N_p^{< \eps}$. In order to conclude, it then suffices to apply the induction assumption since $\dim \Ss_p^{\eps} < \dim F$.
\end{proof}

We conclude this paragraph with a word on the quotient space near principal points. As mentioned earlier, $H \backslash F$ is a Hausdorff topological space, the projection $\mathrm{pr} : F \to H \backslash F$ is continuous but not necessarily smooth. Nevertheless, $\mathrm{pr} : F_0 \to H \backslash F_0$ is smooth in restriction to principal points and it is (tautologically) a Riemannian submersion if we equip the quotient space with the metric $g_{H \backslash F_0}$ such that $g_{H \backslash F_0}(d(\mathrm{pr})(Z),d(\mathrm{pr})(Z)) = g_F(Z,Z)$, for $Z \in TF$. With some slight abuse of notation, we will write $\mu_{H \backslash F}$ for the smooth Riemannian measure on $H \backslash F_0$ (we drop the subscript $0$ in the measure in order to avoid repetition). A point $H \cdot p$ on $H \backslash F$ will also be written $\mathfrak{p}$. Observe that for $f \in C^0(F)$, one has
\begin{equation}
\label{equation:int}
\int_F f \dd \mu_F = \int_{H \backslash F} \left( \int_{H \cdot p}  f(\mathfrak{p},x) \dd \mu_{H \cdot p}(x)\right) \dd \mu_{H \backslash F}(\mathfrak{p}).
\end{equation}
By \eqref{equation:int}, the pushforward $(\mathrm{pr})_*\mu_F$ of the measure $\mu_F$ satisfies for $f \in C^0(H \backslash F)$:
\[
\int_{H \backslash F} f ~(\mathrm{pr})_*\dd \mu_F = \int_F \mathrm{pr}^*f ~\dd \mu_F = \int_{H \backslash F} f(p) \vol(H \cdot p) \dd \mu_{H \backslash F}(\mathfrak{p}),
\]
that is, $(\mathrm{pr})_*\mu_F(\mathfrak{p}) = \vol(H \cdot p) \mu_{H \backslash F}(\mathfrak{p})$. A typical example where $H \backslash F$ is smooth is when $F=G$ is a compact Lie group and $H \leqslant G$ is a subgroup, see \S\ref{ssection:principal} where this is further discussed.

\section{Proof of the Theorem}

\subsection{Ergodicity}

\label{ssection:ergodicity}

We now prove the first part of Theorem \ref{theorem:brin}. We set $E_\star := E_{x_\star}$, where $x_\star$ is an arbitrary periodic point chosen in \S\ref{ssection:lift}. From now on, following Remark \ref{remark:choice}, we assume that an isometric diffeomorphism $E_\star \simeq F$ has been chosen and we will freely identify $E_\star$ with $F$. The general idea of this paragraph is the following: any (smooth) $H$-invariant object (such as a function, set, etc.) defined on $E_\star$ will give rise to a (smooth) flow-invariant object on the whole fiber bundle $E$.

We consider an element $\mathfrak{p} \in H \backslash F$. This corresponds to a closed submanifold $Q_\star(\mathfrak{p}) \subset E_{\star}$ in the fiber over $x_\star$ on which $H$ acts transitively. This submanifold inherits a natural metric $h_\star(\mathfrak{p})$ by restriction of the metric on $E_\star$ to it. Let $\mc{H}$ be the set of all homoclinic orbits to $x_\star$. We then define for $\gamma \in \mc{H}, x \in \gamma$, the set
\begin{equation}
\label{equation:qcirc}
Q_{x}(\mathfrak{p}) := \Hol^c_{x_1 \to x} \circ \Hol^u_{x_\star \to x_1} Q_\star(\mathfrak{p}),
\end{equation}
where $x_1 \in W^u_M(x_\star) \cap \gamma$ is arbitrary. Note that the choice of point $x_1$ on $W^u_M(x_\star)$ is irrelevant in \eqref{equation:qcirc} since $Q_\star$ is invariant by the action of $\rho(\gamma_\star)$. Equivalently, we could have defined $Q_x(\mathfrak{p})$ by using arbitrary stable and then flow-paths joining $x_\star$ to $x$, that is,
\begin{equation}
\label{equation:qcirc2}
Q_{x}(\mathfrak{p}) = \Hol^c_{x_2 \to x} \circ \Hol^s_{x_\star \to x_2} Q_\star(\mathfrak{p}),
\end{equation}
where $x_2 \in W^s_M(x_\star) \cap \gamma$ is arbitrary. This follows from the fact that holonomy along a homoclinic orbit is an element of $H$ and that $Q_\star(\mathfrak{p})$ is $H$-invariant.

We then set
\[
Q(\mathfrak{p}) := \overline{\bigsqcup_{x \in \mc{H}} Q_x(\mathfrak{p})}, .
\]
This set is flow-invariant by construction. 

\begin{lemma}
\label{lemma:q}
The closed set $Q(\mathfrak{p}) \subset E$ is a Hölder-continuous Riemannian fiber bundle $Q(\mathfrak{p}) \to M$, invariant by stable/unstable/flow holonomies, with smooth fiber isometric to $(Q_\star(\mathfrak{p}),h_\star(\mathfrak{p}))$.
\end{lemma}

\begin{proof}
In order not to burden the notation, we write $Q$ instead of $Q(\mathfrak{p})$ in the proof. We divide the proof into three steps: \\

\emph{Step 1:} To start with, let us show that $\pi : Q \to M$ is surjective. Take $x \in M$ and a sequence of points $(x_n)_{n \in Z_{\geq 0}}$ belonging to homoclinic orbits such that $x_n \to x$ (this is always possible by density of homoclinic orbits, see Lemma \ref{lemma:density}). Then, for every $x_n$, we consider $w_n \in Q$ such that $\pi(w_n) = x_n$. Up to extraction, we can assume that $w_n \to w \in E$. Since $Q$ is closed, $w \in Q$ and $\pi(w) = x$. This shows the claim. \\

\emph{Step 2:} We now show that $Q$ is invariant by stable/unstable/flow holonomies, that is, $\Phi_t(Q) = Q$ for all $t \in \R$ and for all $x \in M, y \in W^{s,u}_M(x)$, we have $\mathrm{Hol}^{s,u}_{x \to y}(Q_x) = Q_y$. Let us first argue on $Q|_{\mc{H}}$ (with some slight abuse of notation, we identify $\mc{H}$ and $\sqcup_{\gamma \in \mc{H}} \gamma$ here). Invariance by flow holonomies is immediate from the very construction. We now show invariance by unstable holonomies, the case of stable holonomies being treated similarly. It suffices to show, for $x,y$ belonging to (maybe distinct) homoclinic orbits, the inclusion $\mathrm{Hol}^{u}_{x \to y}(Q_x) \subset Q_y$ since the other inclusion is obtained by reversing the role of $x$ and $y$. Moreover, up to shifting the points $x$ and $y$ by the flow $(\Phi_t)_{t \in \R}$ (since $Q$ is flow-invariant), we can directly assume that $x,y \in W^u_M(x_\star)$. Now, if $w \in Q_x$, we can write $w = \mathrm{Hol}^u_{x_\star \to x} w'$ by \eqref{equation:qcirc}, for some $w' \in Q_\star$. But then 
\[
\mathrm{Hol}^u_{x \to y} w = \mathrm{Hol}^u_{x \to y} \circ  \mathrm{Hol}^u_{x_\star \to x} w' = \mathrm{Hol}^u_{x_\star \to y} w' \in Q_y,
\]
using \eqref{equation:qcirc} once again. To obtain the same statement for the stable holonomy $\mathrm{Hol}^s$, it suffices to use similarly \eqref{equation:qcirc2}. We now show the invariance of $Q$ over $M$ by stable/unstable holonomies. Take $x \in M, y \in W^u_M(x)$, $w \in Q_x$. We consider $w_n \in Q|_{\mc{H}}$ such that $w_n \to w$ and define $x_n := \pi(w_n)$. For every $n \in \Z_{\geq 0}$, we can find points $y_n \in W^u_M(x_n)$ such that $y_n \to y$. Note that, since the weak stable manifold $W^{ws}_M(x_n)$ of $x_n$ is dense in $M$ \cite{Plante-72} and transverse to the unstable manifold $W^{u}_M(x_n)$, we can always assume that $y_n$ lies on a homoclinic orbit, that is $y_n \in W^u_M(x_n) \cap W^{ws}_M(x_n)$. But then, we have $\mathrm{Hol}^u_{x_n \to y_n}w_n \in Q_{y_n}$ by the previous discussion and this converges to $\mathrm{Hol}^u_{x \to y}w$, that is, $\mathrm{Hol}^u_{x \to y}w \in Q_y$ by continuity of the unstable holonomy. Similarly, the same proof works for the stability by stable holonomy. This proves the claim. \\

\emph{Step 3:} We eventually show Hölder-continuity in the basepoint. We take an arbitrary point $x_0 \in \mc{H}$ on a homoclinic orbit and define
\[
\Sigma_{x_0} := \bigcup_{y \in W^u_{M,\delta}(x_0)} W^s_{M,\delta}(y),
\]
where, given $x \in M$, $W^{s,u}_{M,\delta}(x)$ denotes the disk of radius $\delta$ around $x$ in the leaf $W^{s,u}_M(x)$. We can parametrize locally a neighborhood $U$ of $x_0$ by the map (local product structure, see \cite[Proposition 6.2.2]{Fisher-Hasselblatt-19})
\[
\Psi : \Sigma_{x_0} \times (-\delta,\delta) \to U, ~~~~~~~~~(z,t) \mapsto \varphi_t(z).
\]
The map $\Psi$ is Hölder-continuous \cite[Proposition 6.2.2]{Fisher-Hasselblatt-19}. Any point $x \in U$ thus corresponds to a unique triple $(y,z,t)$ where $y \in W^u_{M,\delta}(x_0), z \in W^s_{M,\delta}(y)$ and $|t| < \delta$. We then observe that
\[
Q_x =  \mathrm{Hol}^c_{z \to x} \circ \mathrm{Hol}^s_{y \to z} \circ \mathrm{Hol}^u_{x_0 \to y} Q_{x_0},
\]
and this shows that $Q \to M$ is a fiber bundle over $M$ with fiber isometric to $Q_\star$ since the fiber is isometric to $Q_\star$ over $\mc{H}$ by \eqref{equation:qcirc} and the fibers vary Hölder-continuously in the base variable $x \in M$ since the holonomy maps are Hölder continuous in the base point as explained in \S\ref{ssection:lift}.
\end{proof}

We have thus shown so far that there is a partition
\begin{equation}
\label{equation:stratification}
E = \bigsqcup_{\mathfrak{p} \in H \backslash F} Q(\mathfrak{p}),
\end{equation}
where each $Q(\mathfrak{p}) \to M$ is a flow-invariant Riemannian fiber bundle with Hölder regularity. There is a natural projection $\mathrm{pr} : E \to H \backslash F$ defined as $\mathrm{pr}(w) = \mathfrak{p}$ if $w \in Q(\mathfrak{p})$. We will say that $w$ is \emph{principal} (or, equivalently, that $\mathfrak{p}$ is \emph{principal}) if the isotropy group associated to the $H$-invariant submanifold $Q_\star(\mathfrak{p}) \subset F$ is principal, as defined in \S\ref{ssection:actions}. We denote by $F_0 \subset F$ the (open) set of principal points. We define $E_0 := \mathrm{pr}^{-1}(H \backslash F_0)$ to be the set of principal points.

For every $\mathfrak{p} \in H \backslash F$, each fiber of $Q(\mathfrak{p})$ is smooth (since it is isometric to $Q_\star(\mathfrak{p})$) but it is not clear that $Q(\mathfrak{p}) \to M$ varies smoothly with the base point $x \in M$. When $\mathfrak{p}$ is a principal point, we can show that it does:

\begin{lemma}
\label{lemma:riem}
Let $\mathfrak{p} \in H \backslash F_0$ be a principal orbit. Then $Q(\mathfrak{p}) \to M$ is a smooth Riemannian fiber bundle with fiber diffeomorphic to a closed manifold $Q_0$, independent of $\mathfrak{p} \in H \backslash F_0$.
\end{lemma}

\begin{proof}
Let $p \in E_\star \simeq F$ be a principal point such that $\mathrm{pr}(p) = \mathfrak{p}$. For $\eps > 0$ small enough, let $N_p^{< \eps} := \left\{v \in N_p ~|~ |v| < \eps\right\}$, $S_p := \exp_p(N_p^{< \eps})$ and $\mc{U} := H \cdot S_p$ be a small tubular neighborhood of the orbit $H \cdot p$ as in \S\ref{ssection:actions}. Since $p$ is principal, we get by Theorem \ref{theorem:slice} that
\[
\exp : N^{< \eps} = H \cdot p \times N_p^{< \eps} \ni (z, v) \mapsto \exp_z(v(z)) \in \mc{U},
\]
is an $H$-equivariant\footnote{The $H$-action is $(z,v) \mapsto (h \cdot z, v)$ for $(z,v) \in H \cdot p \times N_p^{< \eps}$.} diffeomorphism where, given $h \in H$ such that $z = h \cdot p$, we set $v(z) := dh_p(v)$. (This is indeed well-defined since $p$ is a principal point and thus the slice representation \eqref{equation:slice-rep} is trivial.) Let $(\e_1,...,\e_k) \in N_p$ be an orthonormal basis, and $\chi \in C^\infty_{\comp}(N_p)$ be a smooth cutoff function with support in $\left\{|v| < \eps\right\}$ and taking value $1$ on $\left\{|v|<\eps/2\right\}$. In the smooth coordinates $(z,v) \in H \cdot p \times N_p^{< \eps}$, we define the functions $f_1, ..., f_k$ by $f_i(z,v) := \chi(v) \langle v, \e_i \rangle$. These are clearly smooth $H$-invariant functions with support in a tubular neighborhood of $H \cdot p$ such that the following holds: setting $\mc{U}' := N_p^{< \eps/2}$, we have $\cap_i \left\{f_i = 0\right\} \cap \mc{U}' = H \cdot p$. Moreover, these functions are non-degenerate on $H \cdot p$, namely $(df_1,...,df_k)$ has maximal rank on $H \cdot p$.

Given a point $x \in M$, we can consider a concatenation of continuous paths all contained in a strong stable/unstable/flow leaf and joining $x$ to $x_\star$. Taking the composition of the stable/unstable/flow holonomies induced by these paths, we obtain an isometry $\mc{I}_x : E_x \to E_{\star}$. We then define $f_i|_{E_x} := \mc{I}_x^*(f_i|_{E_{\star}})$ and the open set $\mc{U}_E' := \cup_{x \in M} \mc{I}_x^{-1} (\mc{U}') \subset E$. We claim that this does not depend on the choice of $us$- and flow-path chosen to connect $x$ to $x_\star$. Indeed, for any other choice $\mc{I}'_x$, $g := \mc{I}'_x \circ \mc{I}_x^{-1}$ is an element of $\mathrm{Isom}(E_\star)$. Moreover, since for every $\mathfrak{p} \in H \backslash F$, the bundle $Q(\mathfrak{p})$ is invariant by stable/unstable/flow holonomies by Lemma \ref{lemma:q}, we get that $g$ preserves the partition
\[
F \simeq E_\star = \bigsqcup_{\mathfrak{q} \in H \backslash F} Q_\star(\mathfrak{q}),
\]
that is $g\left(Q_\star(\mathfrak{p})\right)= Q_\star(\mathfrak{p})$ for all $\mathfrak{p} \in H \backslash F$. Thus, if $f \in C^\infty(F)$ is $H$-invariant, it is constant along all manifolds $Q_\star(\mathfrak{p})$ and thus $g^*f = f$, that is $\mc{I}_x^* f = {\mc{I}'_x}^*f$. (The same argument also shows that the set $\mc{U}_E'$ is well-defined, independently of the choice of isometry $\mc{I}_x$.)

As a consequence, this defines a set of functions $f_i$ on some open flow-invariant subset $\mc{U}_E' \subset E$ which are constant along all submanifolds of the partition \eqref{equation:stratification} by construction. In particular, their restriction to the stable/unstable/flow foliation is smooth. Since $\mc{I}_x : E_x \to E_{\star}$ is an isometry, it is smooth by the Myers-Steenrod Theorem \cite{Myers-Steenrod-39} and thus the functions $f_i$ are fiberwise smooth in the fibers of $E$. Since the stable/unstable/flow foliation and the fiber foliation are all pairwise transverse, we can apply the Journé lemma \cite[Theorem on p.345]{Journe-86} which yields that the $f_i$'s are smooth. In order to conclude, it suffices to observe that $Q(\mathfrak{p}) = \cap_i \left\{f_i = 0\right\} \cap \mc{U}'_E$ and that $(df_1,...,df_k)$ have full rank at $\left\{f_i=0\right\} \cap \mc{U}'_E$ by construction since $\mc{I}_x$ are isometries. Hence, $Q(\mathfrak{p})$ is a smooth submanifold of $E$\footnote{The functions $f_i$ are smooth and well-defined on all of $E$ but one needs to introduce $\mc{U}'_E$ in order to characterize $Q(\mathfrak{p})$ as some zero-set as all the functions $f_i$ vanish away from $\mc{U}_E$.}. The fact that the fiber of $Q(\mathfrak{p})$ does not depend on $\mathfrak{p}$ (as a differentiable manifold) is a direct consequence of the fact that orbits of principal points in $F$ are all diffeomorphic. However, the fibers of $Q(\mathfrak{p})$ \emph{do} depend on $\mathfrak{p}$ as Riemannian submanifolds. Also observe that in this construction, it was important to work with principal orbits because they allow to construct $H$-invariant functions $f_i$ that can then be pushed smoothly to the whole manifold.
\end{proof}

Recall that $\mathrm{pr} : E \to H \backslash F$ is the natural projection and $E_0$ is the set of principal points.

\begin{lemma}
\label{lemma:measure-0}
The map $\mathrm{pr} : E_0 \to H \backslash F_0$ is a smooth Riemannian submersion, where the Riemannian structure on $H \backslash F_0$ is the quotient structure introduced at the end of \S\ref{ssection:actions}. Moreover, $\mu_E(E \setminus E_0) = 0$.
\end{lemma}

\begin{proof}
Let $x \in M$, $\mc{I}_x : E_x \to E_\star$ be an arbitrary isometry obtained by taking holonomies along $us$- and flow-paths joining $x$ to $x_\star$. Let $w \in E_x$ be a principal point. Since this is an open property, any point in a neighborhood of $w$ is also principal. Let $\mathrm{pr}_{\star} : E_\star \simeq F \to H \backslash F$ be the projection. Note that, in restriction to principal points, $\mathrm{pr}_\star : F_0 \to H \backslash F$ is a Riemannian submersion by the discussion at the end of \S\ref{ssection:actions}. We have $\mathrm{pr} = \mathrm{pr}_{\star} \circ \mc{I}_x$ as $Q(\mathfrak{p})$ is invariant by stable/unstable/flow holonomies (by Lemma \ref{lemma:q}) and $\mc{I}_x(w)$ is also principal. Since both $\mc{I}_x$ and $\mathrm{pr}_{\star}$ can be differentiated in the vertical direction, i.e. in the direction of the fiber, respectively at $w$ and $\mc{I}_x(w)$ (note that $\mathrm{pr}_{\star}$ is differentiable at $\mc{I}_x(w)$ in the vertical direction because it is a principal point), we get that $\mathrm{pr}$ can be differentiated in the vertical direction at $w$ and $d (\mathrm{pr})(V) = d(\mathrm{pr}_{\star})(d \mc{I}_x(V))$, for $V \in \V(w)$ a vertical vector. By construction $\mathfrak{p}=\mathrm{pr}(w)$. If $V$ is orthogonal to $Q(\mathfrak{p})$, then $d\mc{I}_x(V)$ is also orthogonal to $Q(\mathfrak{p})$ and $d (\mathrm{pr}_{\star}) :  (T_{\mc{I}_x(w)}Q(\mathfrak{p}))^\bot \to T_{\mathfrak{p}}(H \backslash F_0)$ is an isometry (because $\mathrm{pr}_\star$ is a Riemannian submersion) so the norm of $d (\mathrm{pr})(V)$ is preserved. If $V$ is tangent to $Q(\mathfrak{p})$, then we get $d(\mathrm{pr})(V)=0$. The same argument also shows that the restriction of $\mathrm{pr}$ to the vertical fiber of $w$ is smooth (i.e. can be differentiated infinitely many times in the vertical direction). Moreover, the restriction of the map $\mathrm{pr}$ to $W^{s,u,c}_E(w)$ is constant, equal to $\mathfrak{p}$, so in particular it is smooth. Overall, we see that: 1) $d \mathrm{pr} : T_w Q(\mathfrak{p})^\bot \to T_{\mathfrak{p}}(H \backslash F)$ is an isometry, 2) by the Journé lemma \cite[Theorem on p.345]{Journe-86}, $\mathrm{pr}$ is smooth since it is smooth along a system of transverse foliations. This proves the first part of the Lemma.

It remains to show that principal points have full measure in $E$. By Lemma \ref{lemma:measure}, we know that in the fiber $E_{\star}$, principal points have full measure. Given $x \in M$, we deduce that $E_x \cap E_0 = \mc{I}_x^{-1} (E_{\star} \cap E_0)$ has full measure in $E_x$. Hence:
\[
\int_{E} \mathbf{1}_{E \setminus E_0} \dd \mu_E = \int_{M} \left(\int_{E_x} \mathbf{1}_{E_x \setminus E_0} \dd \mu_{E_x} \right) \dd \mu_M(x) = 0.
\]
\end{proof}

Given $\mathfrak{p} \in H \backslash F_0$, the smooth manifold $Q(\mathfrak{p})$ is a Riemannian fiber bundle over $M$ by Lemma \ref{lemma:riem}. Its fiber $Q_x(\mathfrak{p})$ over $x \in M$ is equipped with a natural metric $h_x$ inherited from the metric on $E_x$ and the fibers $(Q_x,h_x)$ for $x \in M$ are all isometric. In particular, each fiber is equipped with a natural Riemannian measure. Thus, there is a natural smooth flow-invariant measure $\mu_{Q(\mathfrak{p})}$ on $Q(\mathfrak{p})$ obtained by wedging the (pullback of the) measure $\mu_M$ with the Riemannian measure in the fibers. 

\begin{lemma}
\label{lemma:ergodicity}
Let $\mathfrak{p}$ be a principal point. Then the restriction of the flow $(\Phi_t)_{t \in \R}$ to $Q(\mathfrak{p})$ is ergodic with respect to $\mu_{Q(\mathfrak{p})}$.
\end{lemma}

\begin{proof}
For the sake of simplicity, we write $Q = Q(\mathfrak{p})$. Since $\Phi_t : E_x \to E_{\varphi_t x}$ is an isometry, it is also an isometry $Q_x \to Q_{\varphi_t x}$. On each fiber $Q_x$, we can consider the (non-negative) Laplace operator $\Delta_x$ acting on $L^2(Q_x)$ (equipped with the smooth Riemannian volume induced by $h_x$). This operator is also self-adjoint. It has discrete spectrum $\lambda_0 = 0 < \lambda_1 \leq \lambda_2 < ...$ accumulating to $+\infty$. We can decompose $L^2(Q_x) = \oplus_{k \geq 0} \mc{E}_{Q,x}^{(k)}$, where each
\begin{equation}
\label{equation:eq}
\mc{E}_{Q,x}^{(k)} := \ker(\Delta_x-\lambda_k),
\end{equation}
is the finite-dimensional vector space associated to the eigenvalue $\lambda_k$. For each $k \geq 0$, we obtain a well-defined smooth Hermitian vector bundle $\mc{E}^{(k)}_Q \to M$. Moreover there is a natural embedding $\iota_k : C^\infty(M,\mc{E}^{(k)}_Q) \to C^\infty(Q)$ given by $\iota_k f(x,z) := f_x(z)$, where $x \in M, z \in E_x$ and $f_x \in \mc{E}^{(k)}_{Q,x}$ is the section $f$ at the point $x$. We will freely use this identification and drop the notation $\iota_k$. Note that smoothness of $\mc{E}^{(k)}_Q$ follows directly from the local triviality of $E|_{U} \simeq U \times (F,g_F)$ if $U \subset M$ is an open subset. Also observe that the spaces $\mc{E}^{(k)}_Q$ are $\Phi_t$-invariant (i.e. invariant by pullback by $\Phi_t^*$) since the flow is an isometry.

Let $f \in L^2(Q)$ such that $X_Q f = 0$, where $X_Q$ is the restriction of $X_E$ to $Q$. We can write $f = \sum_{k \geq 0} f_k$, where each $f_k$ is a section of $L^2(M,\E^{(k)}_Q)$. Observe that
\[
\|f\|^2_{L^2(Q)} = \sum_{k=0}^{+\infty} \|f_k\|_{L^2(M,\E^{(k)}_Q)}^2,
\]
and each $f_k \in L^2(M,\E^{(k)}_Q)$ satisfies $X_{Q} f_k = 0$. We let $\X_k$ be the restriction of $X_Q$ to sections of $\mc{E}^{(k)}_Q \to M$. It satisfies $\iota_k \X_k = X_Q \iota_k$. Note that $\X_k$ satisfies the assumptions introduced in the beginning of \S\ref{ssection:invariant}: it is a formally self-adjoint differential operator of order $1$, acting as a derivation in the $X_M$-direction. The induced operator $(\Phi_t)_* : \mc{E}_{Q,x}^{(k)} \to \mc{E}_{Q,\varphi_t x}^{(k)}$ is a fiberwise linear isometry now. By the regularity Proposition \ref{proposition:regularity} applied with $\X_k$, we deduce that $f_k$ is smooth. By Proposition \ref{proposition:isomorphism}, this defines by evaluation at $x_\star$ an element $f_k(x_\star) \in \mc{E}^{(k)}_{Q,x_\star}$ (it is an eigenfunction of the Laplacian on $Q_\star$) which is invariant by all elements of the induced representation $\rho^{(k)} : \mathbf{G} \to \mathrm{U}(\E^{(k)}_{Q,x_\star})$. The $\mathbf{G}$-action on $\E^{(k)}_{Q,x_\star}$ is simply given by pullback, namely $\rho^{(k)}(\gamma) u(\bullet) = u(\rho(\gamma)\bullet)$ for all $u \in \E^{(k)}_{Q,x_\star}$. Thus $f_k(x_\star)$ is fixed by all elements of $H$. However, $f_k(x_\star) \in C^\infty(Q_\star)$ and $H$ acts transitively on $Q_\star$, so it implies that $f_k(x_\star)$ is constant. This yields $f_k = 0$ if $k>0$. It remains to deal with $f_0$. The previous argument shows that it is constant on all fibers $Q_{x}$, flow-invariant and smooth, so it descends to a smooth flow-invariant function on $M$. By ergodicity of the flow $(\varphi_t)_{t \in \R}$ on $M$, we get that it is constant. 
\end{proof}

Recall that $\mu_{H \backslash F}$ is the Riemannian measure on $H \backslash F$ as discussed at the end of \S\ref{ssection:actions}. Define the measure $\nu(\mathfrak{p}) := \vol(Q(\mathfrak{p}))\mu_{H \backslash F}(\mathfrak{p})$ on $H \backslash F$, where the volume of $Q(\mathfrak{p})$ is computed with respect to the measure $\mu_{Q(\mathfrak{p})}$. Observe that, as $\mu_M$ was normalized with volume $1$, one has $\vol(Q(\mathfrak{p})) = \vol(Q_\star(\mathfrak{p}))$ (the volume of $Q(\mathfrak{p})$ is equal to the volume of the Riemannian fiber over $x_\star$). A quick computation (similar to the one carried out at the end of \S\ref{ssection:actions}) then shows that $\nu = \mathrm{pr}_*(\mu_E)$. It remains to show the following:

\begin{lemma}
\label{lemma:isomorphism}
There exists a natural isometry $\Psi :  L^2(H \backslash F, \nu) \xrightarrow{\sim} \ker_{L^2}(X_E)$.
\end{lemma}

In other words, given a $H$-invariant $L^2$-function on the fiber $E_{\star}$, there is a unique way of pushing it by the flow in order to obtain an element in $\ker_{L^2}(X_E)$.

\begin{proof}
We start by defining $\Psi$ for smooth functions in $H \backslash F$. Given $u \in C^\infty(H \backslash F)$ (that is, $u$ is smooth on $F$ and $H$-invariant), taking an arbitrary isometry $\mc{I}_x : E_x \to E_{\star}$ obtained from holonomies along $us$- and flow-paths, we set $\Psi u (x) := \mc{I}_x^*u$. As before, we get that $\Psi u \in  \ker_{C^\infty} X_E$, that is, $u$ is a smooth flow-invariant function. Indeed, this follows once again from the Journé Lemma: vertical smoothness is immediate since $\mc{I}_x$ is an isometry while smoothness in the stable/unstable/flow directions follows from the fact that $\Psi u$ is constant along these leaves. 

We now extend $\Psi$ to $L^2(H \backslash F, \nu)$. Since $\mathrm{pr} : E \to H \backslash F$ is a smooth Riemannian submersion except on a set of measure zero in $E$ (by Lemma \ref{lemma:measure-0}) which we may neglect, we may write for all $f \in C^\infty(E)$:
\begin{equation}
\label{equation:desintegration}
\int_E f \dd \mu_E = \int_{H \backslash F} \left(\int_{Q(\mathfrak{p})} f \dd \mu_{Q(\mathfrak{p})}\right) \dd \mu_{H \backslash F}(\mathfrak{p}).
\end{equation}
In particular, we see that
\begin{equation}
\label{equation:decomp-norm}
\|\Psi u\|^2_{L^2(E)} = \int_{H \backslash F} \|\Psi u\|^2_{L^2(Q_{\mathfrak{p}})} \dd \mu_{H \backslash F}(\mathfrak{p}) =  \int_{H \backslash F} |u(\mathfrak{p})|^2  \vol(Q(\mathfrak{p}))\dd \mu_{H \backslash F} = \|u\|^2_{L^2(H \backslash F,\nu)}.
\end{equation}
Hence, by density of $C^\infty(H \backslash F)$ in $L^2(H \backslash F, \nu)$ and using that $\ker_{L^2}(X_E)$ is closed in $L^2(E)$, \eqref{equation:decomp-norm} shows that $\Psi : C^\infty(H \backslash F) \to \ker_{C^\infty}(X_E)$ extends continuously to an isometry $L^2(H \backslash F,\nu) \to \ker_{L^2}(X_E)$. 

In order to obtain that $\Psi$ is onto, it now suffices to show that $\ker_{C^\infty} X_E$ is dense in $\ker_{L^2}(X_E)$. Given $f \in L^2(E)$, we can decompose it in vertical Fourier modes (as we did on $Q$), namely we can write $f = \sum_k f_k$, where each $f_k \in L^2(M,\mc{E}^{(k)}_E)$ is an eigenfunction of the vertical Laplacian on the fibers of $E$, associated to some distinct eigenvalues $\mu_k \geq 0$. The vector field $X_E$ preserves this decomposition. Hence, if $f \in \ker_{L^2}(X_E)$, we get that each $f_k \in L^2(M,\mc{E}^{(k)}_E)$ is also flow-invariant. By Proposition \ref{proposition:regularity}, we get that each $f_k$ is actually smooth. Setting $f_N := \sum_{k \leq N} f_k$, we have that $f_N \in C^\infty(E)$ is flow-invariant, smooth and $f_N \to_{N \to \infty} f$ in $L^2(E)$. This concludes the proof.

\end{proof}

\subsection{Mixing}

\label{ssection:mixing}

We now prove the second part of Theorem \ref{theorem:brin}. We start with the following:
\begin{lemma}
\label{lemma:mixing-cover}
Let $M$ be a smooth manifold endowed with a volume-preserving Anosov flow $(\varphi_t)_{t \in \R}$ and further assume it is mixing. Then the lift of $(\varphi_t)_{t \in \R}$ to any finite cover of $M$ is mixing.
\end{lemma}

\begin{proof}
Let $\pi : \widehat{M} \to M$ be a finite cover and assume that $(\widehat{\varphi}_t)_{t \in \R}$, the lift of $(\varphi_t)_{t \in \R}$, is not mixing. Since it is also volume-preserving (for the lifted measure), we get by the Anosov alternative (see \cite{Anosov-67,Plante-72} or Appendix \ref{appendix}) that it is the suspension of an Anosov diffeomorphism by a constant roof function. In particular, $\mathbb{E}_{\widehat{M}}^s \oplus \mathbb{E}_{\widehat{M}}^u$, the sum of the stable and unstable bundles on $\widehat{M}$, is integrable and since this is a local property, it is also the case for $\mathbb{E}_M^s \oplus \mathbb{E}_M^u$ on $M$. Let $x \in M$, $\widehat{x} \in \widehat{M}$ be a lift of $x$ to $\widehat{M}$, $W^{su}_M(x)$ (resp. $W^{su}_{\widehat{M}}(\widehat{x})$) be the leaf passing through $x$ (resp. $\widehat{x}$) and tangent to $\mathbb{E}_M^s \oplus \mathbb{E}_M^u$. Note that $W^{su}_{\widehat{M}}(\widehat{x})$ is compact. Since $\pi(W^{su}_{\widehat{M}}(\widehat{x}))$ and $W^{su}_M(x)$ coincide near $x$ and are tangent to $\mathbb{E}_M^s \oplus \mathbb{E}_M^u$, they must coincide everywhere, that is $\pi(W^{su}_{\widehat{M}}(\widehat{x})) = W^{su}_M(x)$ and thus $W^{su}_M(x)$ is compact. In turn, by the proof of \cite[Theorem 1.8]{Plante-72}, this implies that $(\varphi_t)_{t \in \R}$ is a suspension with constant roof function, thus contradicting mixing on $M$.
\end{proof}

For all $\mathfrak{p} \in H \backslash F_0$, recall that the fiber of the bundle $Q(\mathfrak{p}) \to M$ is independent of $\mathfrak{p}$ and diffeomorphic to some closed manifold $Q_{0}$, independent of $\mathfrak{p}$ (as a differentiable manifold). 

\begin{lemma}
\label{lemma:mix}
 If the fiber of $Q(\mathfrak{p})$ is not the total space of a Riemannian submersion over the circle $\Ss^1_\theta := \R_\theta/\ell \Z$ (for some $\ell > 0$, equipped with the measure $d\theta^2$), and $(\varphi_t)_{t \in \R}$ is mixing, then the restriction of $(\Phi_t)_{t \in \R}$ to $Q(\mathfrak{p})$ is also mixing. 
\end{lemma}

In particular, if $Q_0$ is not the total space of a fiber bundle over the circle, the conclusions of Lemma \ref{lemma:mix} apply. Here and below, $\mathbf{1}_{E}$ (resp. $\mathbf{1}_Q$) denotes the constant function equal to $1$ on $E$ (resp. on $Q$).

\begin{proof}
Let $f,g \in L^2(Q)$ such that $\langle f, \mathbf{1}_Q\rangle_{L^2} = 0$. Note that, writing $f = \sum_{k \geq 0} f_k, g = \sum_{k \geq 0} g_k$ as before, where $g_k,f_k \in C^\infty(M,\E^{(k)}_Q)$, the condition $\langle f, \mathbf{1}_E\rangle_{L^2} = 0$ is equivalent to $\langle f_0, \mathbf{1}_Q \rangle_{L^2} = 0$. We then have
\[
|\langle e^{tX_Q} f, g \rangle_{L^2}| = \left|\sum_{k \geq 0} \langle e^{t \X_k} f_k, g_k \rangle_{L^2}\right| \leq \sum_{k \geq 0} \|f_k\|_{L^2}\|g_k\|_{L^2} \leq \|f\|_{L^2(E)} \|g\|_{L^2(E)}.
\]
By Lebesgue's dominated convergence, in order to show that $|\langle e^{tX_Q} f, g \rangle_{L^2}| \to 0$ as $t \to +\infty$, it suffices to show that for each $k \in \Z_{\geq 0}$, $|\langle e^{t \X_k} f_k, g_k \rangle_{L^2}| \to 0$.

First of all, we claim that for $k \neq 0$, the operator $\X_k$ is \emph{mixing} on $\E^{(k)}_Q \to M$ in the sense that, given $f_k, g_k \in L^2(M,\mc{E}^{(k)}_Q)$, one has $\langle e^{t\X_k} f_k, g_k \rangle \to_{t \to \infty} 0$. Let us first show that $\X_k$ has no $L^2$-eigenvalues on the imaginary axis. Assume that $i \lambda_0$ is an $L^2$-eigenvalue for $\X_k$, that is, there exists $u \in L^2(M,\E^{(k)}_Q)$ such that $\|u\|_{L^2}=1$ and $\X_k u = i \lambda_0 u$. By Proposition \ref{proposition:regularity}, we get that $u$ is smooth. Given $x \in M, y \in W^s_M(x)$ and considering the holonomy map $(\mathrm{Hol}^s_{x \to y})^* : \mc{E}^{(k)}_{Q,y} \to \mc{E}^{(k)}_{Q,x}$ introduced in \eqref{equation:holonomy} (the action is by pullback of functions), we have that $(\mathrm{Hol}^s_{x \to y})^*u = u$. Note that the $i \lambda_0$ does not appear as it cancels from the transport formula \eqref{equation:holonomy}. We now write $u_\star := u(x_\star) \in \mc{E}^{(k)}_{Q,x_\star}$. As a consequence, given a homoclinic orbit $\gamma$ and using Definition \ref{definition:representation}, we get that $\rho^{(k)}(\gamma)u_\star = \rho(\gamma)^* u_\star = e^{-i \lambda_0 T_{\gamma}} u_\star$, where $\rho^{(k)}$ is the induced representation on $\mc{E}^{(k)}_Q$ (by pullback of functions), $T_\gamma \in \R$ is the time such that $x_2 = \varphi_{T_\gamma}(x_1)$ in Definition \ref{definition:representation}. Taking an arbitrary point $w \in Q_\star$, we observe that $|u_\star(w)| = |u_\star(\rho(\gamma)w)|$ and since $u_\star$ is smooth on $Q_\star$ and $H = \overline{\rho(\mathbf{G})}$ acts transitively on the fiber $Q_{\star}$, we get that $|u_\star|$ is constant. In order to conclude, we can use the following:

\begin{lemma}
Let $(N,g)$ be a closed Riemannian manifold. If there exists a non-constant Laplace eigenfunction of constant modulus, then $(N,g)$ is a Riemannian submersion over the circle.
\end{lemma}

\begin{proof}
Let $u \in C^\infty(N)$ be such that $\Delta_g u = \lambda u$ for some $\lambda > 0$ and assume that $|u|=1$. Then $\Delta |u|^2 = 0 = 2 \lambda |u|^2 - 2|\nabla u|^2$, that is, $|\nabla u| = \lambda^{1/2}$. Hence $u : (N,g) \to (\R/2\pi\Z, \lambda^{-1} d\theta^2)$ is a Riemannian submersion. Note that $(\R/2\pi\Z,\lambda^{-1} d\theta^2)$ is isometric to $(\R/(2\pi\lambda^{-1/2})\Z,d\theta^2)$.
\end{proof}

Hence, by the previous lemma, $u_\star : Q_{\star} \to \Ss^1$ is a Riemannian submersion over the circle and this contradicts the assumptions. Thus $i\lambda_0$ is not an $L^2$-eigenvalue. \\

The operator $\X_k$ is skew-adjoint on $L^2(M,\mc{E}^{(k)}_Q)$. Since it has no $L^2$-eigenvalues on the imaginary axis and admits a spectrum of resonances on anisotropic spaces (i.e. its resolvents $\mathbf{R}_{\pm}(z)$ defined by \eqref{equation:resolvent} admit a meromorphic extension from $\left\{\Re(z) > 0\right\}$ to some half-space $\left\{ \Re(z) > -C s \right\}$ when acting on anisotropic Sobolev spaces $\mc{H}^s_\pm$), we get by Stone's formula (see \cite[Section 5.2]{Cekic-Lefeuvre-20} for further details) that its spectral measure $\dd P(\lambda)$, for $\lambda \in \R$, is absolutely-continuous with respect to the Lebesgue measure and given by
\[
\dd P (\lambda) = -\dfrac{1}{2\pi} \left(\mathbf{R}_+(-i\lambda)+\mathbf{R}_-(i\lambda) \right) \dd \lambda.
\]
Moreover, we have:
\[
\begin{split}
\langle e^{t \X_k} f_k, g_k \rangle_{L^2} & = \int_{-\infty}^{+\infty} e^{i t \lambda} \langle \dd P(\lambda) f_k, g_k \rangle \\
& = - \dfrac{1}{2\pi} \int_{-\infty}^{+\infty} e^{i t \lambda} \langle \left(\mathbf{R}_+(-i\lambda)+\mathbf{R}_-(i\lambda) \right)  f_k, g_k \rangle \dd \lambda = - (2\pi)^{-1} \widehat{F}(-t),
\end{split}
\]
where $F(\lambda) =  \langle \left(\mathbf{R}_+(-i\lambda)+\mathbf{R}_-(i\lambda) \right)  f_k, g_k \rangle$ and $\widehat{\bullet}$ is the Fourier transform. The function $F$ is in $L^1(\R,\dd \lambda)$ by the spectral Theorem (see \cite[Chapter 6, Section 5.3]{Kato-95}) so we get by the Riemann-Lebesgue Theorem that $\widehat{F}(-t) \to 0$ as $t \to \infty$. This argument translating the absence of $L^2$-eigenvalues into mixing is fairly classical, see \cite[Theorem VII.15]{Reed-Simon-80} for instance. \\

It remains to deal with $k=0$. We claim that the only $L^2$-eigenvalue for $\X_0$ is $0$ and corresponds to the one-dimensional eigenspace of constant functions. Indeed, assume that $i \lambda_0$ is an $L^2$-eigenvalue for $\X_0$ and $u \in L^2(M,\E^{(0)}_Q)$ satisfies $\X_0 u = i \lambda_0 u$. Proposition \ref{proposition:regularity} yields that $u$ is smooth and by definition, for all $x \in M$, $u(x,\bullet)$ is constant fiberwise on each connected component of $Q_{x}$ because it is in the kernel of the Laplacian. If the fiber $Q_{\star}$ was connected, $u$ would define equivalently a function $\widetilde{u}$ on the base $M$, satisfying $X_M \widetilde{u} = i \lambda_0 \widetilde{u}$ and we could directly apply the mixing assumption of the flow $(\varphi_t)_{t \in \R}$ to obtain the claimed result. In the general case, defining the equivalence relation for $w,w' \in Q$, $w \sim w'$ if and only if $w$ and $w'$ belong to the same fiber in $E$ and to the same connected component of the fiber $Q$, and setting $Q' := Q/\sim$, we observe that $\pi : Q' \to M$ is a finite connected cover of $M$\footnote{The cover $Q' \to M$ is not necessarily Galois. A sufficient condition for it to be Galois is that $H/H_0$ is Abelian, where $H_0$ is the connected component of the identity in $H$. In this case, it is possible to define a global $H/H_0$ action on $Q'$. It is also Galois when $E$ is a principal bundle, see \S\ref{ssection:principal}.}. The fact that $Q'$ is connected follows from the fact that $(\Phi_t)_{t \in \R}$ is ergodic, hence transitive on $Q$ (i.e. it admits a dense orbit) so $Q$ is connected. The vector field $X_E$ induces a transitive Anosov vector field $X_{Q'}$ on $Q'$ which projects via $\pi$ to $X_M$. Since the smooth function $u$ is constant on each connected component of the fibers of $Q$, it yields a smooth function $\widetilde{u}$ on $Q'$ such that $X_{Q'} \widetilde{u} = i \lambda_0 \widetilde{u}$. By Lemma \ref{lemma:mixing-cover}, the lifted flow on any finite cover of $M$ is mixing, so we get that $\lambda_0 = 0$ and $\widetilde{u}$ is constant, that is, $u$ is constant.

As a consequence, on $L^2(M,\mc{E}^{(0)}_Q)^\bot := L^2(M,\mc{E}^{(0)}_Q) \cap (\C \cdot \mathbf{1}_Q)^\bot$, we get similarly to the case $k\neq0$ that the flow is mixing, namely for $f_0 \in L^2(M,\mc{E}^{(0)}_Q)^\bot, g_0 \in L^2(M,\mc{E}^{(0)}_Q)$, we have $\langle e^{t \X_0} f_0,g_0 \rangle_{L^2} \to 0$. This concludes the proof.
\end{proof}

\begin{remark}
\label{remark:imp}
The same proof also shows that $X_M$ is mixing if and only if there are no non-zero $L^2$-eigenvalues on the imaginary axis.
\end{remark}

The proof of Theorem \ref{theorem:brin} now boils down to gathering the previous Lemmas.

\begin{proof}[Proof of Theorem \ref{theorem:brin}]
Part (a) follows from \S\ref{ssection:ergodicity}, more precisely, Lemmas \ref{lemma:riem}, \ref{lemma:measure-0}, \ref{lemma:ergodicity}, \ref{lemma:isomorphism}. Part (b) follows from \S\ref{ssection:mixing}, Lemma \ref{lemma:mix}.
\end{proof}

We conclude by a simple counter-example when $Q_\star$ fibers over the circle (actually \emph{is} the circle). Given a smooth manifold $M$ equipped with a volume-preserving Anosov flow $(\varphi_t)_{t \in \R}$, we define the bundle $E := M \times \Ss^1$ (where $\Ss^1 = \R/2\pi \Z$) and the flow $\Phi_t(x,\theta) := (\varphi_t(x), \theta + t \text{ mod } 2\pi)$.

\begin{lemma}
\label{lemma:circle}
The flow $(\Phi_t)_{t \in \R}$ is ergodic on $E$ and, in this case, all these sets are equal: $F = \Ss^1 = Q_\star = H$. However, $(\Phi_t)_{t \in \R}$ is not mixing.
\end{lemma}


\begin{proof}
First of all, the flow $(\Phi_t)_{t \in \R}$ cannot be mixing since the smooth function $f(x,\theta) := e^{i\theta}$ satisfies $X_E f = i f$, $\langle f,\mathbf{1}_E\rangle_{L^2}=0$ and thus the correlation
\[
C_t(f,f) := \int_{E} f \circ \Phi_t \cdot \overline{f} ~\dd \mu_E = e^{it} \|f\|_{L^2}^2,
\]
does not converge to $0$. Let us now show ergodicity. Let $H \leqslant \Ss^1$ be the transitivity group. Then $H$ is either equal to $\Ss^1$, in which case the flow is ergodic, or $H$ is finite. Let us show that the latter is impossible. Indeed, if it were the case, then by the first part of Theorem \ref{theorem:brin} we would get a $\Z_k$-bundle over $M$ for some integer $k \in \Z_{\geq 0}$ and the holonomy along every closed orbit $\gamma$ in $M$ would be given by $e^{2i \pi p_\gamma/k}$ for some $p_\gamma \in \left\{0,...,k-1\right\}$. Now, if $T$ is the period of a closed orbit, the holonomy is given by $e^{i T}$ so it suffices to find a closed orbit with length $T$ such that $T \notin (2\pi/k) \Z$. By \cite{Parry-Pollicott-90}, the number of closed orbits in the window $[2\pi n + 1/(3k), 2\pi n + 2/(3k)]$ grows exponentially in $n$, so there is a $n_0$ large enough such that there exists at least one closed orbit whose length is contained in that interval. This provides a contradiction.
\end{proof}

\subsection{The case of principal bundles}

\label{ssection:principal}

We now apply Theorem \ref{theorem:brin} to principal bundles. We let $\pi : P \to M$ be a principal $G$-bundle, where $G$ is a compact Lie group, with right action $R_g : P \to P$ for $g \in G$, equipped with an extension $(\Phi_t)_{t \in \R}$ of the flow $(\varphi_t)_{t \in \R}$ i.e. such that
\begin{equation}
\label{equation:lift-principal}
\forall t \in \R, \forall g \in G, \qquad \pi \circ \Phi_t = \varphi_t \circ \pi, \quad \Phi_t \circ R_g = R_g \circ \Phi_t.
\end{equation}
Each fiber of $P$ is (non-canonically) isometric to $G$. The holonomy formula \eqref{equation:holonomy} can be replaced by a more natural choice. Indeed, we can consider on $P$ an arbitrary principal $G$-connection, and then $\tau_{x \to y} : P_x \to P_y$ (for $x$ close enough to $y$) is defined as the parallel transport along the unique geodesic joining $x$ to $y$ with respect to that connection. Since it is a principal $G$-connection, parallel transport commutes with the right-action, namely $R_g \circ \tau_{x \to y} = \tau_{x \to y} \circ R_g$ for all $g \in G$ and $x,y \in M$ close enough, see \cite[Proposition 3.2]{Kobayashi-Nomizu-96}. By this choice of parallel transport operator, we directly see from \eqref{equation:holonomy} that the stable/unstable holonomies commute with the right-actions, namely if $x \in M, y \in W^{s,u}_M(x)$, then
\begin{equation}
\label{equation:commutation-holonomy}
R_g \circ \mathrm{Hol}^{s,u}_{x \to y} = \mathrm{Hol}^{s,u}_{x \to y} \circ R_g.
\end{equation}

We let $P_\star := P_{x_\star}$ be the fiber over an arbitrary periodic point $x_\star$ used to define homoclinic orbits. By Definition \ref{definition:transitivity}, the transitivity group $H$ is a subgroup of the isometry group $\mathrm{Isom}(P_\star)$ of the fiber $P_\star$. Nevertheless, following Remark \ref{remark:choice}, it is easier in practice to identify the fiber $P_\star$ with the group $G$ itself: for that, we fix an arbitrary element $w_\star \in P_\star$ and then consider the map $\Psi : G \to P_\star, g \mapsto R_g w_\star$. By this identification, for $\gamma \in \mc{H}$, $\Psi^{-1} \rho(\gamma) \Psi$ acts as an isometry of $G$ and by \eqref{equation:commutation-holonomy} it commutes with the right action on $G$ (by itself) so it is a left action on $G$ (by itself) and can thus be identified with an element of the group $G$, namely $\Psi^{-1} \rho(\gamma) \Psi = L_g$ for some $g \in G$. We then define 
\[
H_{w_\star} := \overline{\left\{ g \in G ~|~ \exists \gamma \in \mathbf{G}, L_g = \Psi^{-1} \rho(\gamma) \Psi \right\}}.
\]
In other words, the groups $H_{w_\star} \leqslant G$ and $H \leqslant \mathrm{Isom}(P_\star)$ are simply conjugate by the map $\Psi$.
The group $H_{w_\star}$ is a closed subgroup of the compact Lie group $G$, hence a Lie group. Note that changing the point $w_\star \in P_\star$ by another point $w_\star'$, one gets another subgroup $H_{w_\star'}$ that is conjugate to $H_{w_\star}$ in $G$. In order to simplify notations, we will simply write $H_\star := H_{w_\star} \leqslant G$. Also note that, in the case of a principal bundle, the set of non-principal (or \emph{singular}) points is empty and the quotient space $H_\star \backslash G$ is a smooth manifold.

\begin{corollary}
\label{corollary:principal}
There exists a smooth principal $H_{\star}$-bundle $Q \to M$ such that $w_\star \in Q$, $Q \subset P$ is a flow-invariant subbundle and the restriction of $(\Phi_t)_{t \in \R}$ to $Q$ is ergodic. More generally, one has:
\[
\ker_{L^2} X_P \xrightarrow{\sim} L^2(H_\star \backslash G).
\]
In particular, the principal bundle $P$ admits a reduction of the structure group to $H_{\star}$. If $H_{\star} = G$, then the flow $(\Phi_t)_{t \in \R}$ is ergodic. If $(\varphi_t)_{t \in \R}$ is mixing and $G$ is semisimple, then $(\Phi_t)_{t \in \R}$ is also mixing.
\end{corollary}

Before proving Corollary \ref{corollary:principal}, we formulate a remark in the specific case where $P$ is a frame bundle:

\begin{remark}
Let $\E \to M$ be a smooth Hermitian/real vector bundle of rank $r$ over $M$ and consider an isometric extension $(\Phi_t)_{t \in \R}$ of $(\varphi_t)_{t \in \R}$ to $\E$. Let $P := F \mc{E}$ be the principal frame bundle associated to $\E$ which is a principal $\mathrm{U}(r)$-bundle (or a $\mathrm{SO}(r)$-bundle in the real case). Parry's representation $\rho$ can then be seen as a representation $\rho : \mathbf{G} \to \mathrm{U}(\E_{x_\star})$ (or $ \mathbf{G} \to \mathrm{SO}(\E_{x_\star})$ in the real case) and the transitivity group $H$ is therefore intrinsically defined as a subgroup $H \leqslant \mathrm{U}(\E_{x_\star})$ (or $H \leqslant \mathrm{SO}(\E_{x_\star})$). However, in practice, it is often easier to identify $\mathrm{U}(\E_{x_\star})$ with $\mathrm{U}(r)$, and $H$ with a closed subgroup of $\mathrm{U}(r)$, but this involves some arbitrary choice (an isometry $\C^r \to \E_{x_\star}$) and this shows that $H$, seen as a subgroup of $\mathrm{U}(r)$, is only well-defined up to conjugacy.
\end{remark}

\begin{proof}[Proof of Corollary \ref{corollary:principal}]
The fact that $Q \to M$ is a fiber bundle with fiber isometric to $H_\star$, and that the restriction of $(\Phi_t)_{t \in \R}$ to $Q$ is ergodic, is a direct consequence of Theorem \ref{theorem:brin}. The only point to check is that the induced right-action of $H_\star$ on $Q$ preserves $Q$ (in other words, that $Q$ is indeed equipped with a right-action of $H_\star$ on its fibers and is therefore a principal bundle). On $P_\star \cap Q$, the group $H \leqslant \mathrm{Isom}(P_\star)$ acts on the left, that is $Q \cap P_\star = \left\{ h w_\star ~|~ h \in H \right\}$. Let $\widetilde{h} \in H_\star \leqslant G$ be such that $L_{\widetilde{h}} = \Psi^{-1} h \Psi$ for some $h \in H$. Thus $\widetilde{h} = L_{\widetilde{h}}e_G = \Psi^{-1} h \Psi e_G = \Psi^{-1} h w_\star$. Given an arbitrary point $w = h' w_\star \in Q \cap P_\star$ in the orbit, we then have
\[
R_{\widetilde{h}} w =  h' R_{\widetilde{h}} w_\star = h' \Psi \widetilde{h} = h' \Psi \Psi^{-1} h w_\star = h' h w_\star \in Q \cap P_\star,
\]
so the right action is well-defined on $Q \cap P_\star$. In order to define it on $Q$, it suffices to consider for $x \in M$ an arbitrary isometry $\mc{I}_x : P_x \to P_{x_\star}$ obtained by taking holonomies along $us$- and flow-paths; then for $\widetilde{h} \in H_\star$, using the commutation relation \eqref{equation:commutation-holonomy} and the fact that $\mc{I}_x$ maps $Q_x \to Q_\star$, we obtain that for all $w \in Q \cap P_x$:
\[
R_{\widetilde{h}} w = R_{\widetilde{h}} \mc{I}_x^{-1} \mc{I}_x w = \mc{I}_x^{-1} R_{\widetilde{h}} \mc{I}_x w \in Q.
\]
(Indeed, $\mc{I}_x w \in Q \cap P_\star$, and $R_{\widetilde{h}} \mc{I}_x w \in Q \cap P_\star$ by the previous discussion, and thus $ \mc{I}_x^{-1} R_{\widetilde{h}} \mc{I}_x w \in Q$.) This proves that $Q$ is a principal $H_\star$-bundle. Thus $P$ admits a reduction of its structure group to $H_\star$.

For the last part of the statement relative to mixing, it suffices to observe that $G$ cannot fiber over $\Ss^1$ when $G$ is semisimple. Indeed, since $G$ is connected, this would imply by the exact homotopy sequence that $\pi_1(G)$ surjects onto $\Z$ but compacntess and semisimplicity implies that its fundamental group is finite \cite[Corollary 3.9.4]{Duistermaat-Kolk-00}.
\end{proof}

If $H$ is the transitivity group of $(\Phi_t)_{t \in \R}$, we let $H_0$ be its connected component. We conclude this note by a short lemma showing that, up to taking the finite cover $\widehat{M} := Q/H_0$ of the manifold $M$ (with deck transformation group $H/H_0$), the transitivity group can always be assumed to be connected and given by $H_0$, the connected component of the identity in $H$. Note that, by construction, $Q$ defines a principal $H_0$-bundle over $\widehat{M}$ that we will denote by $\widehat{P}$.

\begin{lemma}
On the finite cover $\widehat{M} := Q/H_0$ of $M$, the extension of the flow $(\Phi_t)_{t \in \R}$ to the bundle $\widehat{P} \to \widehat{M}$ has transitivity group $H_0$. 
\end{lemma}

\begin{proof}
The transitivity group is contained in $H_0$ by construction but the point is to show that it is precisely equal to $H_0$. For that, let $w_\star \in P_\star, \mathfrak{p} = \mathrm{pr}(w_\star)$ and $Q(\mathfrak{p})$ be flow-invariant principal $H$-bundle containing $w_\star$. For the sake of simplicity, we write $Q = Q(\mathfrak{p})$. Since the restriction of the flow to $Q$ is transitive, the bundle $Q \to M$ is a connected principal $H$-bundle and thus $Q/H_0 \to M$ is a finite cover of $M$ with deck transformation group $H/H_0$. (Connectedness of $Q$ follows from the same argument as the one given in the proof of Lemma \ref{lemma:mix}.) We denote it by $\widehat{M}$. It remains to show that the lifted flow $(\widehat{\Phi}_t)_{t \in \R}$ on $\widehat{P}$ over $\widehat{M}$ has transitivity group equal to $H_0$. For the sake of simplicity, let us assume that $H$ has two connected components that is $H/H_0 = \Z_2$, the general case being handled similarly. We write $H_0$ for the identity component of $H$, $H_1$ for the other component. Let $x_1$ and $x_2$ be the two lifts of $x_\star$ to $\widehat{M}$.
Let $\gamma$ be a homoclinic orbit to $x_\star$ on $M$. Then either $\rho(\gamma) \in H_0$ or $\rho(\gamma) \in H_1$. In the first case, $\rho(\gamma)$ maps $H_0$ to $H_0$, in the second case it maps $H_0$ to $H_1$. We let $\widehat{\gamma}$ be the lift of $\gamma$ on $\widehat{M}$, starting at $x_1$. This lift may or may not go through $x_2$, depending on the value of $\rho(\gamma)$. If $\rho(\gamma) \in H_0$, then $\widehat{\gamma}$ is a homoclinic orbit (based in $x_1$) on $\widehat{M}$. If $\rho(\gamma) \in H_1$, then $\widehat{\gamma}$ is an orbit homoclinic to $x_1$ in the past and to $x_2$ in the future. Conversely, if $\widehat{\gamma}$ is a homoclinic orbit to $x_1$ on $\widehat{M}$, then it projects to $\gamma = \pi(\widehat{\gamma})$, a homoclinic orbit to $x_\star$ on $M$ and we have $\rho(\gamma) \in H_0$ by the previous argument.

Let $\rho(g) = h_0 \in H_0$, where $g = \gamma_{p}^{k_p}...\gamma_{1}^{k_1} \in \mathbf{G}$, $\gamma_1,...,\gamma_p \in \mc{H}$, on $M$, and consider the lift of $g$ to $\widehat{M}$ with $\gamma_{1}$ starting at $x_1$. This (non-compact) curve on $\widehat{M}$ is a concatenation of various (non-compact) orbits in $\widehat{M}$, some of them joining $x_1$ (in the past) to $x_2$ (in the future), or vice-versa, or $x_1$ to itself, or $x_2$ to itself. However, the first curve starts at $x_1$ (by assumption) and the last curves ends also at $x_1$ because $\rho(g) \in H_0$.

For the sake of simplicity, let us assume that $g = \gamma_{2} \gamma_{1}$ and $\widehat{g} = \widehat{\gamma_{2}} \widehat{\gamma_{1}}$ where $\widehat{\gamma_{1}}$ connects $x_1$ to $x_2$ and $\widehat{\gamma_{2}}$ connects $x_2$ to $x_1$ (the general argument follows similarly by concatenation). We want to find a sequence of genuine homoclinic orbits $\widehat{\gamma}(n)$ to $x_1$ (i.e. converging in the past and in the future to $x_1$) such that $\widehat{\rho}(\widehat{\gamma}(n)) \to h_0 \in H_0$. For that, we will simply connect $\widehat{\gamma}_{1}$ to $\widehat{\gamma}_{2}$ via the shadowing lemma. By construction, there exist points
\[
p_1^- \in W^u_{\widehat{M}}(x_1) \cap \widehat{\gamma}_1, \quad p_1^+ \in W^s_{\widehat{M}}(x_2) \cap \widehat{\gamma}_1, \quad  p_2^- \in W^u_{\widehat{M}}(x_2) \cap \widehat{\gamma}_2, \quad p_2^+ \in W^s_{\widehat{M}}(x_1) \cap \widehat{\gamma}_1,
\]
such that
\[
\widehat{\rho}(\widehat{g}) = \mathrm{Hol}^s_{p_2^+ \to x_1} \circ \mathrm{Hol}^c_{p_2^- \to p_2^+} \circ  \mathrm{Hol}^{u}_{x_2 \to p_2^-} \circ \mathrm{Hol}^s_{p_1^+ \to x_2} \circ \mathrm{Hol}^c_{p_1^- \to p_1^+} \circ \mathrm{Hol}^u_{x_1 \to p_1^-}.
\]
By \eqref{equation:holonomy}, we have:
\[
\mathrm{Hol}^{u}_{x_2 \to p_2^-} \circ \mathrm{Hol}^s_{p_1^+ \to x_2} = \lim_{t \to + \infty}\widehat{\Phi}_t \circ \tau_{\widehat{\varphi}_{-t}x_2 \to \widehat{\varphi}_{-t}p_2^-} \circ \widehat{\Phi}_{-t} \circ \widehat{\Phi}_{-t} \circ \tau_{\widehat{\varphi}_t p_1^+ \to \widehat{\varphi}_t x_2} \circ \widehat{\Phi}_t,
\]
where $(\widehat{\varphi}_t)_{t \in \R}$ is the flow on $\widehat{M}$.

We let $T_\star$ be the period of the periodic orbit $x_\star$ (on $M$). Note that $\widehat{\gamma}_\star$, the lift of $\gamma_\star$ starting at $x_1$, may have period $T_\star$ or $2T_\star$ according to the value of $\rho(\gamma_\star)$. Consider a subsequence $k_n T_\star$ such that $\widehat{\Phi}_{k_n T_\star}(x_2) \in \mathrm{Isom}(\widehat{P}_{x_2})$ converges to the identity $\mathbf{1}_{\widehat{P}_{x_2}}$ as $n \to \infty$ (and $\widehat{\varphi}_{k_n T_\star}x_2 = x_2$). Then
\[
\mathrm{Hol}^{u}_{x_2 \to p_2^-} \circ \mathrm{Hol}^s_{p_1^+ \to x_2} = \lim_{n \to + \infty} \widehat{\Phi}_{k_n T_\star} \circ \tau_{\widehat{\varphi}_{-k_n T_\star}x_2 \to \widehat{\varphi}_{-k_n T_\star}p_2^-} \circ \tau_{\widehat{\varphi}_{k_n T_\star} p_1^+ \to \widehat{\varphi}_{k_n T_\star} x_2} \circ \widehat{\Phi}_{k_n T_\star}.
\]
The points $\widehat{\varphi}_{k_n T_\star}(p_1^+)$ and $\widehat{\varphi}_{- k_n T_\star}(p_2^-)$ are $\mc{O}(e^{-\lambda k_n T_\star})$-close for some uniform exponent $\lambda > 0$. We let $(x_1;\widehat{\varphi}_{k_n T_\star}(p_1^+)]$ and $[\widehat{\varphi}_{- k_n T_\star}(p_2^-) ; x_1)$ be the half-orbits. By the shadowing lemma, we can concatenate the half-orbits $(x_1;\widehat{\varphi}_{k_n T_\star}(p_1^+)] \sqcup [\widehat{\varphi}_{- k_n T_\star}(p_2^-) ; x_1)$, that is there exists a genuine homoclinic orbit $\widehat{\gamma}(n)$ to $x_1$ which is $\mc{O}(e^{-\lambda k_n T_\star})$ close to the union of orbits $(x_1;\widehat{\varphi}_{k_n T_\star}(p_1^+)] \sqcup [\widehat{\varphi}_{- k_n T_\star}(p_2^-) ; x_1)$. We thus get that the holonomy along $\widehat{\gamma}(n)$ is almost the same as the holonomy along the concatenated orbits $ \widehat{\gamma_{2}} \widehat{\gamma_{1}}$, namely $h_0 \widehat{\rho}(\widehat{\gamma}(n))^{-1} = \mathbf{1}_{\widehat{P}_{x_1}} + \mc{O}(e^{-\lambda k_n T_\star})$, that is $\widehat{\rho}(\widehat{\gamma}(n)) \to h_0$. (In the affine case, see \cite[Lemmas 3.13, 3.14]{Cekic-Lefeuvre-21-1} where this convergence is justified via the Ambrose-Singer formula. A similar argument works in the principal bundle case.). Since $\widehat{H}$ is closed, this shows that $H_0$ is a subgroup of $\widehat{H}$ and thus $\widehat{H} = H_0$.
\end{proof}

\appendix

\section{On volume-preserving Anosov flows}

\label{appendix}

For the sake of completeness, we give in this appendix short proofs of the well-known facts that volume-preserving Anosov flows are ergodic and that they are mixing if and only if they are not the suspension of an Anosov diffeomorphism by a constant roof function. In the case of Anosov diffeomorphisms, similar arguments involving microlocal analysis can be found in \cite[Section 6]{Faure-Roy-Sjostrand-08}, see also \cite[Section 7.2]{Baladi-18}. \\

Let $M$ be a smooth closed connected manifold and $X_M \in C^\infty(M,TM)$ be a smooth vector field generating an Anosov flow $(\varphi_t)_{t \in \R}$ preserving a smooth measure $\mu_M$.

\begin{lemma}
The flow $(\varphi_t)_{t \in \R}$ is ergodic with respect to $\mu_M$.
\end{lemma}

\begin{proof}
Let $u \in L^2(M)$ such that $X_M u = 0$. By Proposition \ref{proposition:regularity}, we get that $u$ is smooth. Hence $u$ is constant along all orbits and stable/unstable manifolds. Given $x \in M$, by the local product structure \cite[Proposition 6.2.2]{Fisher-Hasselblatt-19}, any point in a neighborhood of $x$ can be joined by taking flow- and $us$-paths, so $u$ is locally constant hence constant since $M$ is connected.
\end{proof}

Still in the volume-preserving case, the following is known as the Anosov alternative \cite{Anosov-67,Plante-72}:

\begin{lemma}
The flow $(\varphi_t)_{t \in \R}$ is not mixing if and only if it is the suspension of an Anosov diffeomorphism on a closed connected manifold by a constant roof function.
\end{lemma}

A general statement also holds in the dissipative case, see \cite[Theorem 1.8]{Plante-72} for further details. Following Remark \ref{remark:imp}, a volume-preserving Anosov flow is mixing if and only if there are no $L^2$-eigenvalues on the imaginary axis, except $0$.

\begin{proof}
If the flow is a suspension of an Anosov diffeomorphism with constant roof function, then it is clearly not mixing by the above characterization as there will exist smooth functions in $\ker(X_M-i\lambda)$ (for some $\lambda \in \R \setminus \left\{0\right\}$) of the form $e^{i \lambda \theta}$, where $\theta$ is the $\Ss^1$-variable. Conversely, if it is not mixing, then there exists $u \in L^2(M)$ such that $X u = i\lambda u$, for some $\lambda \in \R \setminus \left\{ 0 \right\}$. By Proposition \ref{proposition:regularity}, $u$ is smooth. Moreover $X|u|^2 = 0$ and thus, by ergodicity (and up to rescaling $u$) we get that $|u|=1$. Hence $u : M \to \Ss^1$ is a smooth map. For any $\theta \in [0,2\pi)$, $N_\theta := u^{-1}(e^{i\theta})$ is a smooth closed submanifold of $M$.

Moreover, $\varphi_t(N_\theta) = N_{\theta + \lambda t}$ so all the manifolds $(N_\theta)_{\theta \in [0,2\pi)}$ are diffeomorphic (and they are all transverse to the flow). We let $F$ be one of the connected components of $N_0$. Given a point $x \in F$ with dense orbit in $M$, there is a time $T > 0$ such that $\varphi_T(x) \in F$. We let $T_0 > 0$ be the smallest positive time such that $\varphi_{T_0}(x) \in F$. Observe that $\varphi_{T_0}(F) = F$ and for all other $y \in F$, $T_0$ is the first positive time such that $\varphi_{T_0}(y)=y$. We can write $T_0 = 2\pi k /\lambda$ for some $k \in \Z \setminus \left\{0\right\}$ and then $N_0 = \bigsqcup_{j=0}^{k-1} \varphi_{2\pi j/\lambda} (F)$ (if $\lambda > 0$) or $N_0 =  \bigsqcup_{j=0}^{-|k|+1} \varphi_{2\pi j/\lambda} (F)$ (if $\lambda < 0$). Setting $f := \varphi_{T_0} : F \to F$ we get that the dynamical system $(M,\varphi_t)$ is conjugate to the flow generated by $\partial_t$ on $F \times [0,T_0]/\sim$, where $(z,T_0) \sim (f(z), 0)$, $z \in F$. It then remains to see that $f$ is Anosov on $F$, but this is immediate as $T F = \mathbb{E}^s_M|_F \oplus \mathbb{E}^u_M|_F$.
\end{proof}

\bibliographystyle{alpha}
\bibliography{Biblio}

\end{document}